\documentclass[11pt]{amsart}

\usepackage[english]{babel}
\usepackage[utf8]{inputenc}
\usepackage[charter]{mathdesign}
\usepackage[T1]{fontenc}
\usepackage[toc,page]{appendix}
\usepackage[a4paper]{geometry}
\usepackage{bbm}
\usepackage{dsfont}
\usepackage{alltt}
\usepackage{verbatimbox}
\usepackage{amsthm,amsfonts}
\usepackage{amsmath}
\usepackage{mathabx}
\usepackage{accents}
\usepackage{etoolbox}
\usepackage{mathtools}
\usepackage{empheq}

\newcommand{\m}[1]{
\ifdefequal{#1}{1}
{\mathbbm{#1}}
{\mathbb{#1}}
}

\newcommand{\q}[1]{\mathcal{#1}}

\newcommand{\ds}{\displaystyle}

\newcommand{\e}{\varepsilon}
\newcommand{\real}{\mathbb{R}}
\newcommand{\complex}{\mathbb{C}}

\newcommand{\weak}{\rightharpoonup}

\newcommand{\be}{\begin{equation}}
\newcommand{\ee}{\end{equation}}

\DeclareMathOperator{\sgn}{\mathrm{sgn}}
\DeclareMathOperator{\Ai}{\mathrm{Ai}}

\DeclareMathOperator{\supp}{\mathrm{Supp}}

\renewcommand{\Re}{\mathop{\mathrm{Re}}}
\renewcommand{\Im}{\mathop{\mathrm{Im}}}

\theoremstyle{plain}
\newtheorem{thm}{Theorem}
\newtheorem*{thm*}{Theorem}
\newtheorem{prop}[thm]{Proposition}

\newtheorem{lem}[thm]{Lemma}
\newtheorem{claim}[thm]{Claim}

\theoremstyle{definition}

\theoremstyle{remark}
\newtheorem{nb}[thm]{Remark}

\makeatletter
\def\blfootnote{\xdef\@thefnmark{}\@footnotetext}
\makeatother

\title{Asymptotics in Fourier space of self-similar solutions to the modified Korteweg-de Vries equation}

\date{}
\author{Simão Correia, Raphaël Côte \and Luis Vega}

\subjclass[2010]{35C06 (primary), 35Q53, 35B40, 45G05} 

\thanks{L. Vega is supported by an ERCEA Advanced Grant 2014 669689 - HADE, by the MEIC project MTM2014-53850-P and MEIC Severo Ochoa excellence accreditation SEV-2013-0323.}

\begin{document}

\allowdisplaybreaks

\begin{abstract}
\noindent We give the asymptotics of the Fourier transform of self-similar solutions to the modified Korteweg-de Vries equation, through a fixed point argument in weighted $W^{1,\infty}$ around a carefully chosen, two term ansatz. Such knowledge is crucial in the study of stability properties of the self-similar solutions for the modified Korteweg-de Vries flow.

\noindent In the defocusing case, the self-similar profiles are solutions to the Painlevé II equation. Although they were extensively studied in physical space, no result to our knowledge describe their behavior in Fourier space. We are able to relate the constants involved in the description in Fourier space with those involved in the description in physical space.

\end{abstract}

\maketitle

\section{Introduction}

We consider the modified Korteweg-de Vries equation:
\begin{align} \label{mkdv} \tag{mKdV}
\partial_t u + \partial_{xxx}^3 u + \e \partial_x (u^3) =0, \quad u: \m R_t \times \m R_x \to \m R.
\end{align}
The signum $\e \in \{\pm 1 \}$ indicates wether the equation is focusing or defocusing.
\eqref{mkdv} solutions enjoy a natural scaling: if $u$ is a solution then
\[ u_\lambda(t,x) := \lambda u(\lambda^3 t, \lambda x) \]
is also a solution to \eqref{mkdv}. We are interested in the self similar solutions of \eqref{mkdv}, that is, solutions which preserve their shape under scaling: in other words, they are solutions of the form
\[ U(t,x) = t^{-1/3} V(t^{-1/3} x) \]
for $t>0$, $x \in \m R$ and where $V: \m R \to \m R$ is the self-similar profile, so that $U_\lambda= U$. After an integration we see that the profile $V$ solves the Painlevé type equation
\begin{align} \label{painleve}
V'' = \frac{1}{3} y V - \e  V^3 + \alpha.
\end{align}
A profile solution to \eqref{painleve} generates a self-similar solution $U$ such that
\begin{gather} \label{def:CI}
U(t) \weak c \delta_{0} + \alpha \mathop{\mathrm{v.p.}} \left( \frac{1}{x} \right) \quad \text{as } t \to 0^+, \quad \text{where} \quad c = \int V(y) dy,
\end{gather}
provided that the mean of $V$ is well defined; we recall that this quantity is preserved by \eqref{mkdv}, and is therefore very relevant.

\bigskip

Self-similar solutions play important roles for the \eqref{mkdv} flow, both for the long time description of solutions. Even for small and smooth initial data, the solutions display a modified scattering where self-similar solutions naturally appear: we refer to Hayashi and Naumkin \cite{HN99,HN01}, which was revisited by Germain, Pusateri and Rousset \cite{GPR16} and Harrop-Griffiths \cite{HaGr16}.

Self-similar solutions and the (mKdV) flow are also relevant as a model for the behavior of vortex filament in fluid dynamics. More precisely, Goldstein and Petrich \cite{GP92} proposed the following geometric flow for the description of the evolution of the boundary of a vortex patch in the plane under the Euler equations:
\[ \partial_t z = - \partial_{sss} z +\partial_s \bar z (\partial_{ss} z)^2, \quad |\partial_s z|^2=1, \]
where $z=z(t,s)$ is complex valued and parametrize by its arctlength $s$ a plane curve which evolves in time $t$. A direct computation shows that its curvature solves the focusing \eqref{mkdv} (with $\e=1$), and self-similar solutions with initial data \eqref{def:CI} corresponds to logarithmic spirals making a corner: this kind of spirals are observed in a number of fluid dynamics phenomenons. We refer to \cite{PV07} and the reference therein for more details. Let us also mention that we were also motivated by the sequence of papers by Banica and Vega \cite{BV08,BV09,BV12,BV13} for related questions, modeled by non linear Schrödinger type equations.

\bigskip

In the defocusing case $\e=-1$, equation \eqref{painleve} actually corresponds to the Painleve II equation, which has its own interest and was intensively studied. Very precise asymptotics where obtained for its solutions. For example, in the case $\e=-1$, $\alpha=0$, for any $\kappa \in \m R$, there exist a unique self similar solution $V_\kappa$ defined for large enough $y \gg 1$ such that
\begin{align} \label{def:v_kappa}
V_\kappa(y) & = \kappa \Ai(y) + O \left( y^{-1/4} e^{-\frac{4}{3\sqrt 3} y^{3/2}} \right) \quad \text{as} \quad y \to +\infty,
\end{align}
where $\Ai$ is the Airy function
\[ \Ai(y) := \frac{1}{\pi} \int_0^{+\infty} \cos \left( \xi^3 + y \xi \right) d\xi. \]
Also, any solution to \eqref{painleve} which tends to 0 as $y \to +\infty$ is one of the $V_\kappa$.
If furthermore $\kappa \in (-1,1)$, $V_\kappa$ is defined on $\m R$ and
\begin{gather}\label{eq:asymVk}
V_\kappa(y) = \frac{2\sqrt{\rho}}{|3y|^{1/4}} \cos\left( \frac{2}{3\sqrt 3} |y|^{3/2} - \frac{3}{2} \rho \ln |y| + \theta \right) + O \left( |y|^{-5/4} \ln |y| \right) \quad \text{as} \quad y \to -\infty \\
\text{where} \quad \rho = \frac{1}{2\pi} \ln \left( \frac{1}{1-\kappa^2} \right) \quad \text{and} \quad \theta = - 3\rho \left( \ln 2 + \frac{1}{4} \ln 3 \right) + \ln \Gamma (i\rho) + \frac{\pi}{2} \sgn \kappa - \frac{\pi}{4}.\nonumber
\end{gather}
($\Gamma$ denotes the Gamma function). Recall for comparison the asymptotics of the Airy function:
\begin{align*} 
\Ai(y) & = \frac{1}{\sqrt{\pi} (3 y)^{1/4}} e^{-\frac{2}{3\sqrt 3} y^{3/2}} + O \left( y^{-5/4} e^{-\frac{2}{3\sqrt 3} y^{3/2}} \right)  & \quad \text{as} \quad y \to +\infty, \\
\Ai(y) & = \frac{1}{\sqrt{\pi} |3y|^{1/4}} \cos\left( \frac{2}{3\sqrt 3} |y|^{3/2} - \frac{\pi}{4} \right) + O \left( |y|^{-5/4} \ln |y| \right) & \quad \text{as} \quad y \to -\infty.
\end{align*}
If $|\kappa|=1$, $V_\kappa$ is still global but is no longer oscillatory as $y \to -\infty$ (it is equivalent to $\sqrt{|y|/2}$ and has a full asymptotic expansion); when $|\kappa| > 1$, $V$ is no longer defined on $\m R$ (it has an infinite number of poles). We refer to the works by Hastings and McLeod \cite{HM80} and Deift and Zhou \cite{DZ95} and the reference therein for the above results, and more (see also \cite{FA83} and the book \cite{FIKN06}). 

In the work of Perelman and Vega \cite{PV07}, related results were obtained in the focusing case $\e =1$, using only ODE techniques. Observe that \eqref{painleve} is rescaled with respect to the way is it presented in those works, and this accounts for the difference in the constants.

\bigskip

However, nothing is known on the Fourier side, even for small $\kappa$ (or small $c$, $\alpha$). The question of the asymptotics of $\hat V$ is natural and interesting by itself. It is also important for the description of solutions to \eqref{mkdv} for large times. Indeed, the Fourier space captures the dispersive effects of the \eqref{mkdv} flow (as it can be seen from the oscillatory behaviour of $\Ai$ or $V_\kappa$ as $z \to -\infty$). This is a key obviously if one wants to study the stability properties of self-similar solutions. 

\bigskip

Here we provide the asymptotics of $\hat V(\xi)$ at high \emph{and} low frequencies $\xi$, for small $(c,\alpha)$. We take our inspiration from PDE techniques, to the contrary of the above mentioned work which relied on ODE or complex analysis methods. One major input of our techniques is that they are amenable to perturbation: this work initiates the study of the (mKdV) flow around self-similar solutions, which will be continued in forthcoming papers.

 We work in weighted spaces based on $L^\infty$: in fact it is convenient to introduce for $k \ge 0$ the space defined  $Z^k$ by
\begin{gather}
Z^k = \left\{ z \in L^\infty(\m R) : \forall \xi >0, \ z(-\xi) = \overline{z(\xi)}, \quad \| z \|_{Z^k} < +\infty \right\} \quad \text{where} \\
\| z \|_{Z^k}  := \|z(\xi)(1+|\xi|^k)\|_{L^{\infty}(\m R)} + \|z'(1+|\xi|^{k+1})\|_{L^{\infty}(0,+\infty)} + \|z'(1+|\xi|^{k+1})\|_{L^{\infty}(-\infty,0)}.
\end{gather}
We emphasize that a finite $\| \cdot \|_{Z^k}$ norm allows for a jump at zero, but with finite  limits at $0^\pm$ (which are conjugate). 

Our main result is the following.

\begin{thm} \label{th1}
Given $\e \in \{ \pm 1 \}$, $k \in \left( \frac{1}{2}, \frac{4}{7} \right)$ and $c,\alpha \in \m R$ with $c^2+\alpha^2<\epsilon_0$ small enough, there exist $A= A(c,\alpha)$ and a real valued function $V \in \q S'(\m R)$ solution to \eqref{painleve} such that
\begin{align}\label{asymp}
\forall \xi >0, \quad  e^{-i\xi^3}\hat V(\xi) = \chi(\xi) e^{i a \ln |\xi|} \left( A + B e^{2ia \ln |\xi|} \frac{e^{-i \frac{8}{9} \xi^3}}{\xi^3} \right) + z(\xi),
\end{align}
where $\chi$ is a $\q C^\infty$ cut-off function such that $\chi(\xi) =0$ if $\xi <1$ and $\chi(\xi)=1$ if $\xi > 2$; the remainder $z \in Z^k$ satisfies 
\begin{gather} \label{est:remainder}
\| z \|_{Z^k} \lesssim |A|, \\
z(\xi) \to c + \frac{3i}{2\pi} \alpha \quad \text{as} \quad \xi \to 0^+,  \label{lim:remainder}
\end{gather}
and the constant $a$ and $B$ are related to $A$ by
\begin{gather}
a = - \frac{3}{4\pi} |A|^2, \quad B = \frac{ -3i\e}{16\pi\sqrt 2} e^{i a \ln 3} |A|^2 A.
\end{gather}
Finally, the map $(c,\alpha) \mapsto A$ is one-to-one onto an adequate neighbourhood of $0 \in \m C$, bi-Lipschitz, and maps $(0,0)$ to $0$.
\end{thm}

\begin{nb}
The symmetry condition in the definition of $Z^k$ reflects the fact that we work with real valued functions (in physical space). For the same reason, the knowledge of $\hat V$ for positive frequencies $\xi>0$ gives a complete description: for $\xi <0$, $\hat V(\xi) = \overline{\hat V(-\xi)}$ and $z(\xi) = \overline{z(-\xi)}$.

In particular, $z$ is continuous if and only if $\alpha=0$, and otherwise has a jump discontinuity of size $\ds \frac{3i}{\pi} \alpha$ at $\xi=0$. Due to \eqref{lim:remainder}, the self-similar solution generated by $V$ satisfies \eqref{def:CI}.
\end{nb}

\begin{nb}
We emphasize that the description of $\hat V$ for large $\xi$ has two terms. Although the second one has decay, its high oscillation means that it is also a leading order term for the derivative $\hat V'$, with decay $1/\xi$ like the first one.

Let us also notice that the parameters $A,B$ and $a$ may vary, but the phase $-8\xi^3/9$ in the second term is completely constrained.  $A$ is related to $c,\alpha$ by an (explicit) integral expression -- see Section \ref{sec:5.1}): it would be nice to have a more computable link.
\end{nb}

\begin{nb}
Performing (lengthy!) computations similar to that in the proofs, one should be able to obtain an asymptotic expansion at any order for high or low frequencies. We will not pursue this question here.
\end{nb}

\begin{nb} \label{nb:complex}
We are interested in real valued solutions to \eqref{painleve} as they are the most relevant for \eqref{mkdv}. However our analysis could be extended to complex valued $V$ (simply dropping the symmetry condition in the definition of $Z^k$). In that case the equation should read
\begin{equation} \label{painleve2}
V'' = \frac{1}{3} x V - \e |V|^2 V + \alpha,
\end{equation}
which corresponds to self similar equation to the gauge invariant \eqref{mkdv}, and the ansatz should look like $\ds c+  \frac{3i}{2\pi} a\sgn(\xi)$ near $\xi =0$, for given $(c,\alpha) \in \m C^2$, and should be written with unrelated constants $A^+, A^-$ instead of $A, \bar A$ for the asymptotics as $\xi \to +\infty$ or $\xi \to -\infty$; and the same for $B$ and $a$.
\end{nb}

\begin{nb}
One natural question is the maximal size of a self-similar solution $V$ defined on $\m R$ so that $V \in \q S'(\m R)$. A conjecture is that, when $\alpha=0$ and the size is measured by the mean $c$, one has a threshold $|c| < \pi/2$ (see \cite{Hoz09}). 

This is not within the scope of our method. Our proofs are done via a fixed point argument, which implies some smallness. We are also limited by our ansatz, with a cut-off function $\chi$ at scale 1. Maybe the result could be sharpened by the use of a cut-off on a scale depending on $(c,\alpha)$.
\end{nb}

As a consequence of the explicit Fourier expansion, we are able to link the profile constructed in Theorem \ref{th1}, with the $V_\kappa$ constructed in physical space in \cite{HM80}.

\begin{prop} \label{prop7}
Fix $\e=-1$ and $\alpha=0$. Then the solution $V$ constructed in Theorem \ref{th1} coincides  with $V_\kappa$ defined in \eqref{def:v_kappa}, where $A$ and $\kappa$ are related via the relation
\begin{align}
|A|^2 =2\ln\left(\frac{1}{1-\kappa^2}\right), \quad \text{and } \mathop{\mathrm{Re}} A \text{ and } \kappa \text{ have same sign.}
\end{align}
\end{prop}

%
%
%
%
%

\section{Outlook of the proof}

\subsection{Fixed point and ansatz}

In Fourier space, equation \eqref{painleve} takes the form
\[ -\frac{i}{3}  \hat V' = \xi^2 \hat V - \e  \q F(|V|^2 V) + \frac{1}{2\pi}\alpha\delta_{\xi=0}. \]
(For convenience, we use the nonlinearity of \eqref{painleve2}, which allows for complex valued $V$, without extra cost on the computations). Denote $v(\xi) = e^{-i \xi^3} \hat V(\xi)$. Then
\begin{align*}
v' & = e^{-i \xi^3} ( \hat V'  -  3i \xi^2 \hat V) = -3i \e e^{-i \xi^3} \q F(|V|^2 V) + \frac{3i}{2\pi}\alpha\delta_{\xi=0}\\
&= -\frac{3i\e}{4\pi^2}  e^{-i \xi^3} \iint_{\eta_1 + \eta_2 + \eta_3 =\xi} e^{i (\eta_1^3 + \eta_2^3 + \eta_3^3)} v(\eta_1) v(\eta_2) \bar v(-\eta_3) d\eta_1 d\eta_2 + \frac{3i}{2\pi}\alpha\delta_{\xi=0}.
\end{align*}
Let us first consider the trilinear operator $I$, which will be central in our analysis:
\begin{gather} \label{def:I}
I(f,g,h)(\xi) := e^{-i  \xi^3} \iint_{\eta_1+\eta_2+\eta_3 = \xi} e^{i (\eta_1^3 + \eta_2^3 + \eta_3^3)} f(\eta_1) g(\eta_2) \bar h(-\eta_3) d\eta_1 d\eta_2.
\end{gather}

We can rewrite $I$ in a more suitable form. Let $\xi = \eta_1 + \eta_2 + \eta_3$, $\eta = \eta_1 + \eta_2$ and $\nu = \eta_1 -  \eta_2$ so that $\eta_3 = \xi - \eta$. We compute
\begin{align*}
\MoveEqLeft \eta_1^3 + \eta_2^3 + \eta_3^3 - (\eta_1 + \eta_2 + \eta_3)^3 \\
& = \eta_1^3 + \eta_2^3 + \eta_3^3 - (\eta_1 + \eta_2)^3 - 3 (\eta_1 +\eta_2)^2 \eta_3 - 3 (\eta_1 +\eta_2) \eta_3^2 - \eta_3^3 \\
& = - 3 \eta_1 \eta_2 (\eta_1 + \eta_2)  - 3 (\eta_1 +\eta_2)^2 \eta_3 - 3 (\eta_1 +\eta_2) \eta_3^2 \\
& = - 3 \left( \frac 1 4 (\eta^2 - \nu^2) \eta + \eta^2 (\xi - \eta) + \eta (\xi - \eta)^2 \right) \\ 
& = -3 \left( \eta \xi^2 - \xi \eta^2 + \frac 1 4 \eta^3 \right) + \frac 3 4 \eta \nu^2.
\end{align*}
Hence
\begin{align*}
\MoveEqLeft I(f,g,h)(\xi) = e^{-i  \xi^3} \iint_{\eta_1+\eta_2+\eta_3 = \xi} e^{i (\eta_1^3 + \eta_2^3 + \eta_3^3)} f(\eta_1)  g(\eta_2) \bar h(-\eta_3) d\eta_1 d\eta_2 \\
& = \iint _{\eta_1+\eta_2+\eta_3 = \xi} e^{-3i \left( \eta \xi^2 - \xi \eta^2+ \frac 1 4 \eta^3 \right) + \frac {3i} 4 \eta \nu^2} f \left( \frac{\eta + \nu}{2} \right)  g \left( \frac{\eta - \nu}{2} \right) \bar h(\eta -\xi) d\eta_1 d\eta_2 \\
& = \frac 1 2 \int_\eta  e^{-3i \left( \eta \xi^2 - \xi \eta^2 + \frac 1 4  \eta^3 \right)} \bar h(\eta- \xi) \left( \int_\nu e^{ \frac {3i} 4  \eta \nu^2} f \left( \frac{\eta + \nu}{2} \right) g \left( \frac{\eta - \nu}{2} \right) d\nu \right) d\eta
\end{align*}
We are thus led to define the operators
\begin{gather} 
J(f,g)(\xi) =  \int_\eta  e^{-3i \Phi(\xi,\eta)} \bar f(\eta- \xi) g(\eta) d\eta, \label{def:J}
\end{gather}
where the phase $\Phi$ 
\begin{gather} \Phi(\xi,\eta) =  \eta \xi^2 - \xi \eta^2 + \frac 1 4  \eta^3, \label{def:Phi} 
\end{gather}
and
\begin{gather}
K(f,g)(\eta) =  \int_\nu e^{ \frac {3i} 4  \eta \nu^2} f \left( \frac{\eta + \nu}{2} \right) g \left( \frac{\eta - \nu}{2} \right) d\nu, \label{def:K}
\end{gather}
so that
\[ I(f,g,h) = \frac{1}{2}J(h,K(f,g)). \]
Back to our problem, our goal is to find a solution to 
\[ v' = -\frac{3i\e}{4\pi^2} I(v,v,v)+\frac{3i}{2\pi}\alpha\delta_{\xi=0}. \]
Equivalently, given $c, \alpha \in \m R$, we define
\begin{gather} \label{def:Psi}
\forall \xi >0, \quad \Psi(v)(\xi)= c +\frac{3i}{2\pi}\alpha-\frac{3i\e}{4\pi^2}\int_0^\xi I(v,v,v)(\eta)d\eta,
\end{gather}
and for $\xi <0$, $\Psi(v)(\xi) = \overline{\Psi(v)(-\xi)}$. We are looking for a fixed point of $\Psi$ (and $v(0^+)=c+\frac{3i}{2\pi}\alpha$); we will consider it of the form
\[ v = S + z, \]
where $S$ is our ansatz and $z$ is small in some adequate functional space.

We first have to find a good ansatz $S$ for the self-similar solution $v$ on the Fourier side, and we will prove the existence of such a solution $V$ using a fixed-point argument.

We consider a smooth cut-off function $\chi$ with $\chi\equiv 0 $ for $\xi<1$ and $\chi \equiv 1$ for $\xi >2$.
In order to obtain a real valued self-similar solution, we impose that $S(-\xi) = \bar S(\xi)$, which means that we may focus on the region $\xi>0$. Then we consider the two term ansatz
\begin{gather} \label{eq:ansatz}
\forall \xi >0, \quad S(\xi)=\chi(\xi)e^{ia\ln|\xi|}\left(A + B e^{2ia\ln|\xi|}\frac{e^{i\beta \xi^3}}{\xi^3}\right),\end{gather}
with constants $A,B \in \m C$ and $\beta, a \in \m R$ to be adjusted. Observe that $S(0)=0$ (and so $z(0^+)=c+\frac{3i}{2\pi}\alpha$).

To see if $S$ is a good approximation of the self-similar solution, we shall compute $\Psi(S)$ and then compare with $S$. The first term 
\[ e^{ia \ln |\xi|} \]
in the ansatz $S$ comes from the following heuristics. It seems natural to look first at the $\Psi(\m 1)$, because constants for $v$ correspond to the Airy function for $V$, which is a solution to the linear part \[ \Ai'' = \frac{1}{3} y \Ai \]
 of the Painlevé equation \eqref{painleve}. In fact, the leading term $\Psi(\m 1)$ presents slow oscillations, of the form $e^{ia \ln |\xi|}$ for large $\xi$ (this can be seen by computing the leading term $I(\m 1,\m 1, \m 1)$: this is not done here, but would follow, in a simplified way, from the computations done in Sections 3 and 4). 

Then if we use this improved approximation, we are led to compute the leading term of $\Psi(e^{ia \ln |\xi|})$: at least formally and for a correct choice of $a$, it is $e^{ia \ln |\xi|}$ itself!

Now when doing a rigorous proof, it turns out that derivatives are absolutely needed to control the errors. But when we consider the derivatives $\partial_\xi I(S,S,S)$, we see another term at leading order, which is given by the second, highly oscillating, term 
\[ e^{3ia \ln |\xi|} \frac{e^{i \beta \xi^3}}{\xi^3} \]
in the ansatz. This second term can not be avoided, and requires that we do a distinct analysis for low and high frequencies. Fortunately, the introduction of this second term in the ansatz does not lead to a different asymptotic development for $\Psi(S)$ and we are able to complete the proof with the two terms ansatz \eqref{eq:ansatz}.

This procedure is actually quite analogous to the Picard iteration scheme: one starts with a suitable initial function and computes various iterates of $\Psi$. In our method, we start with $\m 1$ and compute three iterates of $\Psi$. Thankfully, the error between the third iterate and the true solution can be controlled and a fixed point can be applied.

One of the main difficulties in completing this program is to obtain a correct estimation of the remainder terms. In the integrals involved in \eqref{def:J}-\eqref{def:Phi}-\eqref{def:K}, we see that the phases are \emph{quadratic} (or cubic), which naturally leads to stationary phase estimates. This means a rather slow decay, and also the need to develop efficient bounds on the errors on the stationary phase. This should be done preferably in $L^\infty$ based spaces: indeed, we have pointwise estimates on the main order terms, and the problem is critical in some sense (the ansatz has no decay at infinity for example), so that we can not afford to lose information.

This is in sharp contrast with the analogous problem for the nonlinear Schrödinger equation. In that case, the phases appearing in the integrals are linear, and thus are never stationary: the analysis is much simpler.

\bigskip

When matching the behavior of $S$ and $\Psi(S)$ at $\xi =0$ and $\xi = \pm \infty$, the constants $A,B, a, \alpha$ and $c$ are linked. On the other side, it turns out that $\beta$ does not depend on any other constant, and is in fact universal:
\[ \beta = - \frac{8}{9}. \]
However, to see how this phenomenon occurs, we will pursue the computations for arbitrary $\beta$. In several steps, the shape of the expansions obtained will depend on $\beta$. To avoid unnecessary computations, we will always assume that 
\[ \beta \in (-1, -1/2). \] 
This allows to perform the computations without dichotomy in the expression of the expansion.

\subsection{Organisation of the proofs and notations}

Our analysis will be done in spaces based on weighted $W^{1,\infty}$ in $v$: this is coherent with the first term of the ansatz $S$, which has no decay at infinity. In our goal to \emph{construct} a solution, we do not aim at dealing with rough data,  it is sufficient for us to work with relatively strong norms.

One major difficulty in the proofs is that we are not able to close an argument in a functional space that contains both $S$ and $z$. The reason which we detail below, essentially comes from the fact that $S$ (and self-similar solutions), although smooth, has poor decay properties at infinity.

More precisely, in the process of closing the fixed point argument, a very delicate game is to be played with the errors in stationary phase arguments. The control of the errors is technically challenging, in particular, we cannot allow the use of too many derivatives. 

On the one hand, we absolutely need to control at least one derivative on $z$, as weighted $L^\infty$ spaces are not sufficient to capture the dispersive effects (see for example Lemma \ref{lem:nonstat}, \ref{lem:stat} and \ref{lem:regsing}). 

But on the other hand, it turns out that $S$ does not quite belong to the right weighted space, which is essentially given by \eqref{est:remainder}. Indeed for $S$, we only have the decay
\[ \| S \|_{L^\infty} \lesssim |A|, \quad \| (1+|\xi|) S' \|_{L^\infty} \lesssim |A| + |B|,\]
so that we miss a power $k >1/2$.

So for terms in $S$ only, we will involve the second derivative of $S$ (essentially via integration by parts), to compensate the lack of decay of $S$. However, if we were to compute with second derivative of $z$, too, then again the decay of $S''$ would not be sufficient. 

This is why, at many places in the following sections, we will prove two estimates on the same quantity, one meant to be used for the ansatz $S$ and the other for the remainder $z$. 

\bigskip

The multiplicity of the norms involved also has an impact on the exposition in the proof, in particular the Landau notation $O$. So as to keep the expression as simple as possible, we adopt the following convention: during the proof of an estimate, the implicit constant involved in $O$ is allowed to depend multilinearly on the norms appearing in the factors of the right hand side of the final estimates. For example, in the course of proving the estimate
\[ \| B(f,g) \|_N \le C \| f \|_{N_1}  \| g \|_{N_2}, \]
(where $B$ is a bilinear map and $N$, $N_1$ and $N_2$ are norms), the bound $L(g)(\xi) = O(\xi)$ (where $L$ is linear) means that there exist an absolute constant $C$ such that
\[ |L(g)(\xi)| \le C |\xi| \| g \|_{N_2}, \]
in the neighborhood in $\xi$ considered. So as to avoid ambiguity, we will specify clearly what estimate is being proved at each step. The same convention holds for the symbol $\lesssim$. 

We will also write  $f \sim g$ for two complex valued functions $f$ and $g$ if $|f/g|$ is bounded below and above by some strictly positive constants. We will use the notation $\sgn$ for the signum function, which can take values in $\{\pm 1\}$ or in $\{ \pm \}$, depending on the context.

Finally, we point out the remainder $z$ in Theorem \ref{th1} may present a jump discontinuity at $\xi=0$. This means that, in the estimates meant for $z$, an integration by parts will yield a boundary term at this point. Sometimes, for convenience of notation, we simply include this boundary term in the integral and interpret $z'(0)$ as a Dirac delta distribution.

\bigskip

Section 3 is devoted to estimates on $J$. In Section 4 we compute the precise asymptotics for $K(S,S)$ and $I(S,S,S)$. We prove Theorem \ref{th1} and Proposition \ref{prop7} in Section 5.

\section{Preliminary estimates} \label{sec:3}

\begin{lem}
Let $\lambda > 0$. Then
\begin{align} 
\left| \int_\lambda^\infty e^{i \eta^2} d\eta  - \frac{1}{2i \lambda} \right| & \le \frac{1}{\lambda^3}, \\
\left| \int_\lambda^\infty e^{i \eta^2} d\eta  - \frac{\sqrt \pi}{2} e^{i \pi/4}  \right| & \le \lambda.
\end{align}
\end{lem}

\begin{proof}
For estimate for large $\lambda$, we do two integrations by parts:
\begin{align*}
\int_\lambda^\infty e^{i \eta^2} d\eta & = \int_\lambda^\infty \frac{2i\eta}{2i\eta} e^{i \eta^2} d\eta
 = \frac{1}{2i \lambda} + \int_{\lambda}^\infty \frac{e^{i\eta^2}}{2i\eta^2} d\eta \\
& = \frac{1}{2i \lambda} - \frac{1}{4 \lambda^3} -  \int_{\lambda}^\infty \frac{e^{i\eta^2}}{12\eta^4} d\eta.
\end{align*}
Then a crude triangular inequality yield the bound
\[  \frac{1}{4 \lambda^3} +  \int_{\lambda}^\infty \frac{d\eta}{12\eta^4} = \left( \frac{1}{4} + \frac{1}{36} \right) \frac{1}{\lambda^3} \le \frac{1}{\lambda^3}. \]
For the estimate for small $\lambda$, we simply use $| e^{i \eta^2}| \le 1$ and 
\[ \int_\lambda^\infty e^{i \eta^2} d\eta - \frac{\sqrt \pi}{2}  e^{i \pi/4} =   \int_\lambda^\infty e^{i \eta^2} d\eta - \int_0^\infty e^{i \eta^2} d\eta = - \int_0^\lambda e^{i \eta^2} d\eta. \qedhere \]
\end{proof}

\begin{lem} 
[Fundamental bounds]\label{lem:f} For any $\xi\neq 0$,
\begin{align}\label{eq:fond2}
\left| \int e^{i \xi \eta^2} g(\eta) d\eta - \sqrt\frac{\pi}{|\xi|} e^{i \frac{\pi}{4} \sgn(\xi)} g(0) \right| &  \lesssim \frac{\|g\|_\infty}{\sqrt{|\xi|}}  + \frac{\|g'\|_\infty}{|\xi|} 
\end{align}
Furthermore, if there exists $R>0$ such that  $\supp g \subset [-\xi R, \xi R]$, then 
\begin{equation}\label{eq:fond1}
\left| \int e^{i \xi \eta^2} g(\eta) d\eta - \sqrt\frac{\pi}{|\xi|} e^{i \frac{\pi}{4} \sgn(\xi)} g(0) \right|  \le  C\frac{\ln|\xi|}{|\xi|}\|g'\|_\infty,\quad C=C(R).
\end{equation}
\end{lem}

\begin{proof}
We assume $\xi >0$, the other case is similar.
\begin{align*}
& \int e^{i \xi \eta^2} g(\eta)  d\eta - \sqrt\frac{\pi}{\xi} e^{i \frac{\pi}{4}} g(0) = \int e^{i \xi \eta^2} (g(\eta) - g(0)) d\eta \\
& =  \int_{\eta=0}^\infty \int_{\nu=0}^\eta e^{i \xi \eta^2} (g'(\nu) - g'(-\nu)) d\nu d\eta \\
& = \int_{\nu=0}^{\infty} (g'(\nu) - g'(-\nu))  \int_{\eta= \nu}^\infty e^{i \xi \eta^2} d\eta d\nu \\
&  = \frac{1}{\sqrt{\xi}} \int_{\nu=0}^{\infty} (g'(\nu) - g'(-\nu))  \left( \int_{\sqrt{\xi} \nu}^\infty e^{i\mu^2} d\mu  \right) d\nu\end{align*}
(with $\mu = \sqrt{\xi} \mu$). 
We split the previous integral at $b$. As
\[ \int_{\sqrt{\xi} \nu}^\infty e^{i\mu^2} d\mu =  \frac{\sqrt \pi}{2} e^{i \pi/4} + O( \sqrt{\xi} \nu), \]
there holds
\begin{align*}
\MoveEqLeft \int_{0}^{b} (g'(\nu) - g'(-\nu)) \left( \int_{\sqrt{\xi} \nu}^\infty e^{i\mu^2} d\mu  \right) d\nu
=  \int_0^b (g'(\nu) - g'(-\nu)) \left( \frac{\sqrt \pi}{2} e^{i \pi/4} + O( \sqrt{\xi} \nu) \right) d\nu \\
& =  \frac{\sqrt \pi}{2} e^{i \pi/4} (g(b) - g(-b)) + O \left( \sqrt{\xi} \int_{|\nu| \le b} |\nu g'(\nu)| d\nu \right)
\end{align*} 
Also,
\[ \left| \int_{\eta= \sqrt{|\xi|} \nu}^\infty e^{i \eta^2} d\eta \right| \le \frac{C}{\sqrt{|\xi|}\nu} , \]
so that
\[
\left| \int_{b}^\infty (g'(\nu) - g'(-\nu))  \int_{\sqrt{|\xi|} \nu}^\infty e^{i \mu^2} d\mu d\nu \right| \le \frac{1}{\sqrt{|\xi|}} \int_{|\nu| \ge b} |g'(\nu)| \frac{d\nu}{\nu}. 
\]
The second estimate now follows from choosing $b=|\xi|^{-1/2}$. For the first estimate, it is necessary to refine the estimate for $\nu>b$: in fact, since
\[ \left| \int_{\eta= \sqrt{|\xi|} \nu}^\infty e^{i \eta^2} d\eta  - \frac{1}{2i\sqrt{|\xi|}\nu}\right| \le \frac{C}{|\xi|^{3/2}|\nu|^3}, \]
one has
\begin{align*}
 \int_{b}^\infty g'(\nu) \int_{\sqrt{|\xi|} \nu}^\infty e^{i \mu^2} d\mu d\nu & = \frac{1}{2i\sqrt{|\xi|}} \int_b^\infty \frac{g'(\nu)}{\nu} d\nu + O\left(\int_b^\infty \frac{|g'(\nu)|}{|\xi|^{3/2}|\nu|^3}d\nu \right) \\ & = -\frac{1}{2i\sqrt{|\xi|}}\frac{g(b)}{b} + \frac{1}{2i\sqrt{|\xi|}}\int_b^\infty\frac{g(\nu)}{\nu^2}d\nu +  O\left(\frac{\|g'\|_\infty}{|\xi|^{3/2}b^2}\right) \\&= O\left(\frac{\|g\|_\infty}{b|\xi|^{1/2}} + \frac{\|g'\|_\infty}{|\xi|^{3/2}b^2} \right)
\end{align*}
Thus
\begin{align*}
\left| \int e^{i \xi \eta^2} g(\eta) d\eta - \sqrt\frac{\pi}{|\xi|} e^{i \frac{\pi}{4} \sgn(\xi)} g(0) \right| &  \lesssim\int_{|\nu| \le b} |g'(\nu) \nu| d\nu + \frac{\|g\|_\infty}{b|\xi|} + \frac{\|g'\|_\infty}{|\xi|^{2}b^2} \\
& \qquad  +  \frac{1}{\sqrt{|\xi|}} | g(b) - g(-b)| \\ &\lesssim \|g'\|_\infty b^2 + \frac{\|g\|_\infty}{b|\xi|} + \frac{\|g'\|_\infty}{|\xi|^{2}b^2}  + \frac{\|g\|_\infty}{\sqrt{|\xi|}}.
\end{align*}
The claimed estimate now follows from choosing $b=|\xi|^{-1/2}$.
\end{proof}
The following four lemmas concern to the asymptotic behaviour of the parametric integral
\[
J(f,g)(\xi)=\int e^{-3i\Phi(\xi,\eta)}f(\eta)g(\eta-\xi)d\eta.
\] 

\begin{lem}\label{lem:primeiro}
Fix $k=(1/2)^+$. If $|\xi|<2$,
\begin{equation} \label{est:l1_1}
\left|J(f,g)(\xi)\right|\lesssim \|(1+|\eta|^{k+1})f\|_{L^\infty}\|g\|_{L^\infty}
\end{equation}
and
\begin{align}
\left|J(f,g)(\xi)\right|&\lesssim \left(\|(1+|\eta|^{1/2})f\|_{L^\infty} + \|f'|\eta|^{3/2}\|_{L^\infty(\{|\eta|>1\})} \right) \nonumber \\
&\qquad \times \left(\|g\|_{L^\infty} + \|g'|\eta|\|_{L^\infty(\{|\eta|>1\})}\right)  \label{est:l1_2}
\end{align}
\end{lem}

\begin{proof}
\emph{Proof of estimate \eqref{est:l1_1}.} It is direct:
\begin{align*}
\left|\int e^{-3i\Phi(\xi,\eta)}f(\eta)g(\eta-\xi)d\eta\right| & \le \int_{|\eta|<1} \|f\|_{L^\infty(\{|\eta|<1\})}\|g\|_{L^\infty} d\eta \\
& \qquad + \int_{|\eta|>1} |\eta|^{-1-k} \|f|\eta|^{k+1}\|_{L^\infty(\{|\eta|>1\})}\|g\|_{L^\infty} d\eta
\end{align*}

\emph{Proof of estimate \eqref{est:l1_2}.}
We write
\[  \int e^{-3i\Phi(\xi,\eta)}f(\eta)g(\eta-\xi)d\eta= \int_{|\eta| \le 10} + \int_{|\eta| \ge 10}. \]
The first term can be bounded directly:
\[ \left|\int_{|\eta| \le 10} e^{-3i \eta^3/4} f(\eta) g(\eta-\xi) d\eta\right|\lesssim \|f\|_\infty\|g\|_\infty \int_{|\eta|<10}d\eta = O(1). \]
For $|\eta| \ge 10$, there is no stationary points and we can do an integration by parts. Notice that $\partial_\eta \Phi(\xi,\eta) \ge c \eta^2$ where $c$ is uniform in $|\xi| \le 2$; Denote
\[ G(\xi,\eta) = e^{-3i \Phi(\xi,\eta)} f(\eta) g(\eta-\xi). \]
Then
\begin{align*}
\int_{|\eta| \ge 10} G(\xi,\eta) d\eta = \int_{|\eta| \ge 10} G(0,\eta) d\eta + \int_0^\xi \int_{|\eta| \ge 10} \partial_\zeta G(\zeta,\eta) d\eta d\zeta.
\end{align*}
Now for $|\zeta| \le |\xi| \le 2$,
\begin{align*}
\int_{|\eta| \ge 10} \partial_\zeta G(\zeta,\eta) d\eta = &\int_{|\eta| \ge 10} e^{-3i \Phi(\zeta,\eta)} f(\eta) i (\partial_\zeta \Phi) g(\eta- \zeta) d\eta\\ & - \int_{|\eta| \ge 10} e^{-3i \Phi(\zeta,\eta)} f(\eta)g'(\eta-\zeta) d\eta 
\end{align*}

The second integral is bounded directly. For the first integral, we do an integration by parts:
\[
\int_{|\eta| \ge 10} e^{-3i \Phi(\zeta,\eta)} f(\eta) i (\partial_\zeta \Phi) g(\eta- \zeta) d\eta = \int_{|\eta| \ge 10} e^{-3i\Phi(\zeta,\eta)} \partial_\eta \left( \frac{\partial_\zeta \Phi}{\partial_\eta \Phi}  f(\eta)  g(\eta- \zeta)\right) d\eta.
\]
Since
\[  \partial_\eta \left( \frac{\partial_\zeta \Phi}{\partial_\eta \Phi} \right)  = \frac{\partial_{\eta \zeta}^2 \Phi}{\partial_\eta \Phi} - \frac{\partial_\zeta \Phi \partial_{\eta\eta}^2 \Phi}{(\partial_{\eta} \Phi)^2} = O \left( \frac{\eta}{\eta^2} \right) + O\left( \frac{\eta^2 \eta}{\eta^{2\cdot 2}} \right) = O(|\eta|^{-1}), \quad \left|\frac{\partial_\zeta \Phi}{\partial_\eta \Phi}\right| \lesssim 1 \]
and
\[ f(\eta) = O(|\eta|^{-1/2}),\quad f'(\eta)=O(|\eta|^{-3/2}), \quad g(\eta-\zeta)=O(1), \quad g'(\eta-\zeta) = O(|\eta|^{-1}), \]
we estimate
\begin{align*}
&\int_{|\eta| \ge 10} e^{-3i\Phi(\zeta,\eta)} \partial_\eta \left( \frac{\partial_\zeta \Phi}{\partial_\eta \Phi}\right)  f(\eta)  g(\eta- \zeta) d\eta +\int_{|\eta| \ge 10} e^{-3i\Phi(\zeta,\eta)} \frac{\partial_\zeta \Phi}{\partial_\eta \Phi} \partial_\eta \left( f(\eta)  g(\eta- \zeta) \right) d\eta \\&\lesssim O(1) + \int_{|\eta|\ge 10} \left(|f'(\eta)||g(\xi-\eta)|+|f(\eta)||g'(\xi-\eta)|\right)d\eta\\& \lesssim O(1) + \int_{|\eta|\ge 10} \frac{1}{|\eta|^{3/2}} d\eta= O(1).
\end{align*}
Hence we can expand 
\[ \int_{|\eta| \ge 10} G(\xi,\eta) d\eta =  \int_{|\eta| \ge 10} e^{-3i \eta^3/4} f(\eta) g(\eta) d\eta + O(|\xi|) \]
and the claimed estimate follows.
\end{proof}

We now focus on estimates for $J$ in the case $|\xi|>2$. For this, we do not give a global result, but rather we split between various regions, as it will be needed in Section 4.

First of all, let us remark that 
\[ \Phi(\xi,\eta) = \xi^3 P(\eta/\xi),\quad P(X) := X - X^2 + \frac{1}{4} X^3.  \]
The polynomial $P$ has two non-degenerate critical points $X=2/3$ and $X=2$:
\begin{equation}\label{eq:estat2}
P(2)=0,\quad P'(2)=0,\quad P''(2) =1,
\end{equation}
\begin{equation}\label{eq:estat23}
P(2/3) = 8/27,\quad  P'(2/3)=0,\quad  P''(2/3) = -3/2.
\end{equation}
Around these points, we will use a stationary phase argument (see Lemma \ref{lem:stat}) using the estimates from Lemma \ref{lem:f}. On the other hand, we want to handle functions $f$ with singularities at the origin, which means that one should tread lightly around $X=0$. For the remaining regions, the integrand presents no singularity and the phase has no stationary points. Hence we may use integration by parts to obtain strong decay estimates (see Lemma \ref{lem:nonstat}).

Let $\phi$ be a radial cut-off function such that $\phi(r)=1$ for $0 \le r \le 1$ and $\phi(r) =0$ for $r \ge 7/6$. Let $\varphi(r) = \phi(r)-\phi(2r)$ so that $\varphi(r) =1$ if $7/12 \le r \le1$ and $\varphi (r)=0$ if $r \le 1/2$ or $r \ge 7/6$.

Define
\begin{align}\label{eq:particaounidade}
\varphi_1(r) & = \phi(8r/3), \nonumber\\
\varphi_2(r) & = \varphi(4r/3), \nonumber\\
\varphi_3(r) & = \varphi(2r/3), \\
\varphi_4(r) & = \varphi(r/3),\nonumber \\
\varphi_5(r) & = 1 - \phi(r/3).\nonumber
\end{align}
One checks that
\[ \varphi_1 + \varphi_2 + \varphi_3 + \varphi_4 + \varphi_5 =1, \]
and $\varphi_2(r) = 1$ if $7/16 \le r \le 3/4 $ and $2/3$ belongs to that interval; $\varphi_3(r) =1$ if $7/8 \le r \le 3/2$ and 1 belongs to that interval; $\varphi_4(r) =1$ if $7/4 \le r \le 3$ and $2$ belongs to that interval.

Define
\[
J_j(f,g)(\xi)=\int e^{-3i\Phi(\xi,\eta)}f(\eta)g(\eta-\xi)\varphi_j(\eta/\xi)d\eta
\]
so that $J=J_1+J_2+J_3+J_4+J_5$.

\begin{lem}
[Non-stationary regions] \label{lem:nonstat}Fix $k>1/2$. For $|\xi|>2$ and $j=3,5$, we have
\begin{align}
|J_j(f,g)(\xi)|& \lesssim |\xi|^{-5/2}\ln|\xi|\left(\|f|\eta|^{1/2}\|_{L^\infty(\{|\eta|>1\})} + \|f'|\eta|^{3/2}\|_{L^\infty(\{|\eta|>1\})}\right) \nonumber \\
& \qquad\times\left( \|g\|_{L^\infty} + \|(1+|\eta|)g'\|_{L^\infty(\real\setminus\{0\})} \right) \label{est:l2_1}
\end{align}
and
\begin{align}
|J_j(f,g)(\xi)| &\lesssim |\xi|^{-k-1}\left(\|f|\eta|^{k+1}\|_{L^\infty(\{|\eta|>1\})} + \|f'|\eta|^{k}\|_{L^\infty(\{|\eta|>1\})}\right) \nonumber\\
& \qquad\times\left( \|g\|_{L^\infty} + \|(1+|\eta|)g'\|_{L^\infty(\real\setminus\{0\})}  \right) \label{est:l2_2}
\end{align}

\end{lem}

\begin{proof}
Over the supports of $\varphi_3$ and $\varphi_5$, the phase $\Phi$ is not stationary: $|\eta|^2\lesssim|\partial_\eta\Phi|$. We then integrate by parts:
\[
J_j(f,g)=\int e^{-3i\Phi(\xi,\eta)}\partial_\eta\left(\frac{1}{-3i\partial_\eta \Phi(\xi,\eta)}f(\eta)g(\eta-\xi)\varphi_j(\eta/\xi)\right)d\eta.
\]
The claimed estimates now follow from applying the bounds of $f$ ang $g$ directly (notice that a boundary term appears at $\eta=\xi$ because $g$ may be discontinuous at $0$; however, this term poses no extra difficulty).
\end{proof}

\begin{lem}[Stationary regions] \label{lem:stat} Fix $k>1/2$. For $|\xi|>2$,
\begin{align}
\MoveEqLeft J_2(f,g)(\xi) + J_4(f,g)(\xi)\label{eq:regstat} \\
& = \sqrt{\frac{\pi}{3|\xi|}}\left(e^{-i\sgn(\xi)\pi/4}f(2\xi)g(\xi)\sqrt{2} + e^{i\sgn(\xi)\pi/4}f(2\xi/3)g(-\xi/3)e^{-8i\xi^3/9}\frac{2}{\sqrt{3}}\right) + R(\xi) \nonumber
\end{align}
where
\begin{equation}\label{eq:restoln}
|R(\xi)|\lesssim \frac{\ln|\xi|}{|\xi|^{5/2}}\left(\|f|\xi|^{1/2}\|_\infty + \|f'|\xi|^{3/2}\|_\infty\right)\left(\|g\|_\infty + \|g'|\xi|\|_\infty\right)
\end{equation}
and
\begin{equation}\label{eq:restok1}
|R(\xi)|\lesssim \frac{1}{|\xi|^{k+1}}\left(\|f|\xi|^{k+1}\|_\infty + \|f'|\xi|^{k}\|_\infty\right)\left(\|g\|_\infty + \|g'|\xi|\|_\infty\right).
\end{equation}

\end{lem}
\begin{proof}

We first obtain the asymptotics for $J_4(f,g)$ with the error estimate \eqref{eq:restoln}. Recalling \eqref{eq:estat2}, we may define $\psi_1$ such that
\[ P(\psi_1(\mu)) = \mu^2,\]
$\psi_1$ is a diffeomorphism on $[-c_1,d_1]$ ($c_1,d_1>0$) to its image $[3/2,4]$ with $\psi'_1 \ge \delta >0$ and $\psi(0) = 2$ and $\psi_1'(0)=\sqrt 2$. We can also extend it to a diffeomorphism $\real \to \real$. 
Define the change of variable $\eta = \psi_1(\mu/\xi) \xi$ for $\eta \in [\xi, 3\xi]$, 
\[ \Phi(\xi,\eta) = \xi^3 P(\eta/\xi) = \xi^3 P(\psi_1(\mu/\xi)) = \xi \mu^2. \]
Now
\begin{align*}
\MoveEqLeft \int e^{-3i \Phi(\xi,\eta)} f(\eta) g(\eta-\xi) \varphi_4(\eta/\xi) d\eta \\
& = \int e^{-3i \xi \mu^2}  f(\psi_1(\mu/\xi)\xi) g((\psi_1(\mu/\xi)-1)\xi) \varphi_4(\psi_1(\mu/\xi)) \psi'_1(\mu/\xi) d\mu \\
& =: \int e^{-3i \xi \mu^2}  h(\xi,\mu) d\mu.
\end{align*}
Notice that due to $\varphi_4$, $h(\xi, \cdot)$ has compact support  inside $\{ \mu \mid 3/2 \le \psi_1(\mu/\xi) \le 21/6 \} \subset [-c_1 \xi,d_1 \xi]$. In particular, on the support of $h$, $\psi_1(\mu/\xi)-1 \ge 1/2$ and we have, for any $\mu \in \mbox{supp }h(\xi,\cdot)$,
$$ |h(\mu,\xi)| \le |f(\psi_1(\mu/\xi)\xi)|\|g\|_\infty\|\phi_4\|_\infty\|\psi_1'\|_{L^\infty(-c_1,d_1)}\le \frac{C\|f|\xi|^{1/2}\|_\infty}{|\psi_1(\mu/\xi)\xi|^{1/2}}\le \frac{C}{|\xi|^{1/2}}, $$
and similarly
 $$ |\partial_\mu h(\xi,\mu)| \le  \frac{C}{|\psi_1(\mu/\xi)\xi|^{3/2}}+\frac{C}{|\psi_1(\mu/\xi)\xi|^{1/2}|(\psi_1(\mu/\xi)-1)\xi|}+\frac{C}{|\psi_1(\mu/\xi)\xi|^{1/2}|\xi|}\le \frac{C}{|\xi|^{3/2}} . $$
Hence, using the fundamental bound \eqref{eq:fond1}, we get that the error is $O(|\xi|^{-5/2}\ln |\xi|)$ and
\begin{align*}
\int e^{-3i \Phi(\xi,\eta)} f(\eta) g(\eta-\xi) \varphi_4(\eta/\xi) d\eta & = \sqrt{\frac{\pi}{3|\xi|}} e^{-i \pi /4 \sgn(\xi)} h(\xi,0) + O(|\xi|^{-5/2}\ln|\xi|) \\
& = \sqrt{\frac{\pi}{3|\xi|}} e^{-i \pi/4 \sgn(\xi)} f(2\xi) g(\xi)\sqrt{2} + O(|\xi|^{-5/2}\ln|\xi|)
\end{align*}
Similarly, for $\varphi_2$: as $P(2/3) = 8/27$, $P'(2/3)=0$ and $P''(2/3) = -3/2$, we can consider the diffeomorphism $\psi_2 : [-c_2, d_2] \to [1/3;5/6]$ such that
\[ P(\psi_2(\mu)) = 8/27 - \mu^2, \]
with $\psi_2(0) = 2/3$, $\psi_2'(0) = 2/\sqrt{3}$. Then with $\eta = \psi_2(\mu/\xi)\xi$,
\[ \Phi(\xi,\eta) = \xi^3 P(\eta/\xi) = \frac{8}{27} \xi^3 - \xi \mu^2. \]
And we extend $\psi_2$ into a diffeomorphism on $\real$. Then the same computations show that
\begin{align*}
\MoveEqLeft \int e^{-3i \Phi(\xi,\eta)} f(\eta) g(\eta-\xi) \varphi_2(\eta/\xi) d\eta \\
&  = \sqrt{\frac{\pi}{3 |\xi|}} e^{i\pi/4 \sgn(\xi)} f(2\xi/3) g(-\xi/3) e^{-i 8\xi^3/9} \psi'_2(0)+ O(|\xi|^{-5/2}\ln|\xi|)
\end{align*}
The asymptotics with the error estimate \eqref{eq:restok1} follow from applying \eqref{eq:fond2} instead of \eqref{eq:fond1}.
\end{proof}

\begin{lem}[Singular region]\label{lem:regsing}
 Fix $k>1/2$. For $|\xi|>2$,
\begin{equation}\label{eq:singk1}
|J_1(f,g)(\xi)|\lesssim \frac{1}{|\xi|^{k+1}}\left(\|f|\xi|^{1/2}\|_\infty + \|f'|\xi|^{3/2}\|_\infty\right)\left(\|g|\xi|^k\|_\infty + \|g'|\xi|^{k+1}\|_\infty\right)
\end{equation}
and
\begin{equation}\label{eq:singln}
|J_1(f,g)(\xi)|\lesssim \frac{\ln|\xi|}{|\xi|^{2}}\left(\|f\|_\infty + \|f'|\xi|\|_\infty\right)\left(\|g\|_\infty + \|g'|\xi|\|_\infty\right)
\end{equation}
\end{lem}

\begin{proof}
We write
\begin{align*}
\int  e^{-3i \Phi(\xi, \eta)} & f(\eta) g(\eta - \xi) \varphi_1(\eta/\xi) d\eta 
 = T_{1,1} + T_{1,2} + T_{1,3} \qquad  \text{with} \\
T_{1,1} & =  \int_{\eta \ge 0} e^{-3i \Phi(\xi, \eta)} f(\eta) g(\eta - \xi) \phi(|\xi|^{2} \eta) d\eta \\
T_{1,2} & =  \int_{\eta \le 0} e^{-3i \Phi(\xi, \eta)} f(\eta) g(\eta - \xi) \phi(|\xi|^{2} \eta) d\eta \\
T_{1,3} & =  \int e^{-3i \Phi(\xi, \eta)} f(\eta)  g(\eta - \xi) (\varphi_1(\eta/\xi)  -  \phi(|\xi|^{2} \eta)) d\eta
\end{align*}
\emph{Proof of estimate \eqref{eq:singk1}.} We have
\[
|T_{1,1}| + |T_{1,2}|\lesssim \int_0^{|\xi|^{-2}} |\eta|^{-1/2}|\xi|^{-k}d\eta = O(|\xi|^{-k-1})
\]
and, for $T_{1,3}$, we apply integration by parts: since
\[
\left|\frac{1}{\partial_{\eta}\Phi(\xi,\eta)}\right|\lesssim \frac{1}{|\xi|^2},\quad \left|\frac{\partial_{\eta\eta}\Phi(\xi,\eta)}{(\partial_{\eta}\Phi(\xi,\eta))^2}\right|\lesssim \frac{1}{|\xi|^3},
\]
we obtain
\begin{align*}
 T_{1,3} & = \int e^{-3i \Phi(\xi,\eta)} \partial_\eta \left( \frac{1}{3\partial_{\eta} \Phi(\xi,\eta)} f(\eta)g(\eta-\xi) (\varphi_1(\eta/\xi) - \phi(\xi^{2} \eta)) \right) d\eta \\
& = \int_{|\xi|^{-2}}^{|\xi|/2} \frac{1}{\xi^2} O \left( \frac{1}{|\xi|^{1+k} |\eta|^{1/2}} + \frac{1}{|\eta|^{3/2}|\xi|^k} + \frac{1}{|\xi|^{1+k} |\eta|^{1/2}} + \frac{1}{|\xi|^{1+k} |\eta|^{1/2}} \right) d\eta \\
& \qquad + \int_{|\xi|^{-2}/2}^{2|\xi|^{-2}} \frac{1}{\xi^{2+k}} O \left( \frac{|\xi|^{2}}{|\eta|^{1/2}} \right) d\eta \\
& = O \left(  |\xi|^{-k-1} \right)
\end{align*}

\emph{Proof of estimate  \eqref{eq:singln}.} The bounds now write:
\[ 
|T_{1,1}| + |T_{1,2}|\lesssim\int_0^{|\xi|^{-2}} d\eta = O(|\xi|^{-2}).
\]
and
\begin{align*}
 T_{1,3} & = \int e^{-3i \Phi(\xi,\eta)} \partial_\eta \left( \frac{1}{3\partial_{\eta} \Phi(\xi,\eta)} f(\eta) g(\eta-\xi) (\varphi_1(\eta/\xi) - \phi(\xi^{2} \eta)) \right) d\eta \\
& = \int_{|\xi|^{-2}}^{|\xi|/2} \frac{1}{\xi^2} O \left( \frac{1}{|\xi|} + \frac{1}{|\eta|}\right) d\eta + \int_{|\xi|^{-2}/2}^{2|\xi|^{-2}} \frac{1}{\xi^2} O \left( |\xi|^{2} \right) d\eta = O(|\xi|^{-2}\ln|\xi|). \qedhere
\end{align*}

\end{proof}

\section{Asymptotics for $\Psi(S)$} \label{sec:4}

We start with the asymptotics for $K(S,S)$.

\begin{lem}\label{lem:desenvolveISS}
There exists $D\in \complex$, such that if $|\eta| \le 10$,
\begin{align} \label{est:l7_1}
\left| K(S,S)(\eta) - e^{i \sgn(\eta) \pi /4} \sqrt{\frac{4\pi}{3}}  \frac{|A|^2}{\sqrt{|\eta|}} - D \right| & \lesssim |\eta|.
\end{align}

If $\eta \ge 10$,
\begin{align}  \label{est:l7_2}
\left| K(S,S)(\eta) - e^{i \pi /4} \sqrt{\frac{4\pi}{3}} A^2 \frac{e^{i a \ln (|\eta|^2/4)} }{\sqrt{|\eta|}} \right| & \lesssim \frac{1}{|\eta|^{2}} .
\end{align}

If $\eta \le -10$,
\begin{align}  \label{est:l7_3}
\left| K(S,S)(\eta) - e^{-i \pi /4} \sqrt{\frac{4\pi}{3}} \overline{A}^2 \frac{e^{- i a \ln (|\eta|^2/4)} }{\sqrt{|\eta|}} \right| & \lesssim \frac{1}{|\eta|^{2}} .
\end{align}

For the derivative, if $|\eta|<10$,

\begin{align}  \label{est:l7_4}
\left| \partial_\eta K(S,S)(\eta) - e^{i \sgn(\eta) \pi /4} \sqrt{\frac{4\pi}{3}} \frac{|A|^2}{|\eta|^{1/2}\eta}\right| & \lesssim \frac{1}{|\eta|}.
\end{align}

If $\eta>10$, there exists a bounded function $A_+$ such that
\begin{align}  \label{est:l7_5}
\left| \partial_\eta K(S,S)(\eta) - \frac{A_+(\eta)}{|\eta|^{3/2}}\right| & \lesssim \frac{\ln|\eta|}{|\eta|^3} .
\end{align}

If $\eta<-10$, there exists a bounded function $A_-$ such that
\begin{align}  \label{est:l7_6}
\left| \partial_\eta K(S,S)(\eta) - \frac{A_-(\eta)}{|\eta|^{3/2}}\right| & \lesssim \frac{\ln|\eta|}{|\eta|^3} .
\end{align}
\end{lem}

\begin{proof}
Let 
\[ \tilde S(\eta,\nu) = S \left( \frac{\eta+\nu}{2} \right) S \left( \frac{\eta-\nu}{2} \right) . \]

\emph{Proof of \eqref{est:l7_2}.} In this case, $\eta \ge 10$. With $\mu = \sqrt{|\eta|} \nu$, we have
\begin{align*}
K(S,S)(\eta) & = \frac{1}{\sqrt{|\eta|}} \int e^{3i \mu^2/4} \tilde S \left(\eta, \frac{\mu}{\sqrt{| \eta|}} \right) d \mu \\
\sqrt{|\eta|} K(S,S)(\eta) & = \int_{|\mu| \le |\eta|^{3/2}/2} +  \int_{| \eta|^{3/2}/2 \le |\mu|} = T_{2,1}+T_{2,2}.
\end{align*}
We start with the estimate for $T_{2,1}$. Then $|\mu| \le |\eta|^{3/2}/2$, so that $|\eta \pm \mu/\sqrt{|\eta|}| \ge |\eta| (1  - 1/2) \ge 4$ and in that region
\begin{align*}
\tilde S(\eta,\mu/\sqrt{\eta}) &= e^{-ia\ln4}e^{ia\ln|\eta^2-\mu^2/\eta|}\left(A + B e^{2ia\ln|(\eta+\mu/\sqrt{|\eta|})/2|}\frac{e^{i\beta (\eta+\mu/\sqrt{\eta})^3}}{(\eta+\mu/\sqrt{\eta})^3}\right)\\&\quad \times\left(A + B e^{2ia\ln|(\eta-\mu/\sqrt{|\eta|})/2|}\frac{e^{i\beta (\eta-\mu/\sqrt{\eta})^3}}{(\eta-\mu/\sqrt{\eta})^3}\right)
\end{align*}
The terms with at least one $B$ are estimated directly: for example,
\begin{align*}
&\left|\int_{|\mu|\le |\eta|^{3/2}/2} e^{3i\mu^2/4}e^{ia\ln(\eta^2-\mu^2/\eta)} e^{2ia\ln|(\eta+\mu/\sqrt{|\eta|})|}\frac{e^{i\beta (\eta+\mu/\sqrt{\eta})^3}}{(\eta+\mu/\sqrt{\eta})^3} d\mu\right|\\ \le & \int_{-|\eta|^{3/2}/2}^{|\eta|^{3/2}/2} \frac{1}{(\eta+\mu/\sqrt{\eta})^3} d\mu = O(|\eta|^{-3/2}).
\end{align*}
We treat the term with $A^2$: we have
\begin{align*}
e^{i a \ln (\eta^2 - \mu^2/|\eta|) } & = e^{i a\ln |\eta|^2}  \exp \left( ia \ln \left( 1 - \frac{\mu^2}{|\eta|^3} \right) \right) = e^{i a\ln |\eta|^2} \left( 1 + \frac{\mu^2}{|\eta|^3} \phi \left( \frac{\mu}{|\eta|^{3/2}} \right) \right), 
\end{align*}
for some function $\phi$ which is smooth on $[-1/2,1/ 2]$ and such that $\| \phi \|_{W^{1,\infty}([-1/2,1/2]} \le C$. 
Hence ($z = \mu/|\eta|^{3/2}$)
\begin{align*}
\int e^{3i\mu^2/4} e^{ia\ln|\eta^2-\mu^2/\eta|} d\mu& =  e^{i a\ln |\eta|^2} \int_{|\mu| \le |\eta|^{3/2}/2} e^{3i\mu^2/4}  \left( 1 + \frac{\mu^2}{|\eta|^3} \phi \left( \frac{\mu^2}{|\eta|^3} \right) \right) d\mu \\
& =  e^{i a\ln |\eta|^2} |\eta|^{3/2} \int_{|z| \le 1/2} e^{3i|\eta|^3 z^2/4}  \left( 1 + z^2 \phi ( z) \right) dz 
\end{align*}
Now,
\begin{gather*}
\int_{|z| \le 1/2} e^{3i|\eta|^3 z^2/4} dz = \sqrt\frac{4\pi}{3|\eta|^{3}} e^{i \pi /4} + O(|\eta|^{-3}) \\
\begin{aligned}
\MoveEqLeft \int_{|z| \le 1/2} e^{3i|\eta|^3 z^2/4} z^2 \phi \left( z \right) dz  = \frac{2}{3i|\eta|^3} \int_{|z| \le 1/2} \frac{3i}{2} |\eta|^3 z e^{3i|\eta|^3 z^2/4} z \phi ( z ) dz \\
& = \frac{2}{3i|\eta|^3} \left( [ e^{3i|\eta|^3 z^2/4} z \phi \left( z \right) ]_{-1/\sqrt 2}^{1/\sqrt 2} - \int_{|z| \le 1/2} e^{3i|\eta|^3 z^2/4} (z \phi \left( z \right))' dz \right) = O \left( \frac{1}{|\eta|^3} \right)
\end{aligned}
\end{gather*}
So that
\[ T_{2,1} = A^2 \sqrt{\frac{4\pi}{3}} e^{i \pi/4} e^{i a \ln (|\eta|^2/4)} + O( |\eta|^{-3/2}). \]

We estimate $T_{2,2}$: here, it is important to decompose $S$ as
\[
S(\xi)=S_1(\xi) + S_2(\xi),\quad S_1(\xi)=Ae^{ia\ln|\xi|}\chi(\xi) + \overline{A}e^{-ia\ln|\xi|}\chi(-\xi).
\]
Notice that, for all $\xi\in\real$,
\[
|S_1(\xi)|\lesssim |A|,\quad |S_1'(\xi)|\lesssim \frac{|A|}{1+|\xi|},\quad |S_1''(\xi)|\lesssim \frac{|A|}{1+|\xi|^2},\quad 
|S_2(\xi)|\lesssim \frac{|B|}{1+|\xi|^3}.
\]
Writing
\[
\tilde S_1(\eta,\nu)=S_1\left(\frac{\eta+\nu}{2}\right)S_1\left(\frac{\eta-\nu}{2}\right),
\]
we have
\begin{align*}
T_{2,2}&= \int_{|\mu|>|\eta|^{3/2}/2}e^{3i\mu^2/4} \tilde S_1\left(\eta,\frac{\mu}{\sqrt{\eta}}\right)d\mu \\&\quad+ O\left(\int\int_{\substack{|\mu|>|\eta|^{3/2}/2\\|\eta\pm\mu/\sqrt{\eta}|>2}} \frac{1}{1+|\eta+\mu/\sqrt{\eta}|^3} + \frac{1}{1+|\eta-\mu/\sqrt{\eta}|^3}d\mu\right)\\&=\int_{\mu \ge |\eta|^{3/2}/2} \mu e^{3i\mu^2/4} \frac{1}{\mu} \tilde S_1 \left( \eta, \frac{\mu}{\sqrt{\eta}} \right) d\mu + O(|\eta|^{-3/2})\\ & = \int_{\mu \ge |\eta|^{3/2}/2} \mu e^{3i\mu^2/4} \frac{1}{\mu} \tilde S_1 \left( \eta, \frac{\mu}{\sqrt{\eta}} \right) d\mu + O(|\eta|^{-3/2}) \\ &= \left[   \frac{2}{3i} e^{3i\mu^2/4} \frac{1}{\mu} \tilde S_1 \left( \eta, \frac{\mu}{\sqrt{\eta}} \right)  \right]_{-|\eta|^{3/2}/2}^{|\eta|^{3/2}/2}  \\
& \quad - \frac{2}{3i} \int_{\mu \ge |\eta|^{3/2}/2} e^{\frac{3i\mu^2}{4}} \left(-  \frac{1}{\mu^2} \tilde S_1 \left( \eta, \frac{\mu}{\sqrt{\eta}} \right) + \frac{1}{\mu \sqrt{\eta}}  \partial_\nu \tilde S_1 \left( \eta, \frac{\mu}{\sqrt{\eta}} \right) \right) d\mu + O(|\eta|^{-3/2}) \\
& =  \frac{2}{3i\sqrt{\eta}} \int_{\mu \ge |\eta|^{3/2}/2} e^{\frac{3i\mu^2}{4}} \frac{1}{\mu}  \partial_\nu \tilde S_1 \left( \eta, \frac{\mu}{\sqrt{\eta}} \right) d\mu + O(|\eta|^{-3/2})
\end{align*}
With another integration by parts, 
\begin{align*}
\MoveEqLeft \int_{\mu \ge |\eta|^{3/2}/2} e^{\frac{3i\mu^2}{4}} \frac{1}{\mu}  \partial_\nu \tilde S_1 \left( \eta, \frac{\mu}{\sqrt{\eta}} \right) d\mu = \frac{2}{3i} \left[ e^{\frac{3i\mu^2}{4}} \frac{1}{\mu^2} \partial_\nu \tilde S_1 \left( \eta, \frac{\mu}{\sqrt{\eta}} \right) \right]_{-|\eta|^{3/2}/2}^{|\eta|^{3/2}/2} \\
& \quad -  \frac{2}{3i}  \int_{\mu \ge |\eta|^{3/2}/2} e^{\frac{3i\mu^2}{4}}\left(-  \frac{2}{\mu^3} \partial_{\nu} \tilde S_1 \left( \eta, \frac{\mu}{\sqrt{\eta}} \right) + \frac{1}{\mu^2 \sqrt{\eta}}  \partial_{\nu \nu}^2 \tilde S_1 \left( \eta, \frac{\mu}{\sqrt{\eta}} \right) \right) d\mu = O ( |\eta|^{-2} )
\end{align*}
This concludes the proof of  estimate \eqref{est:l7_2}. The proof for \eqref{est:l7_3} is similar. 

\bigskip

\emph{Proof of \eqref{est:l7_1}.} We now turn to the case when $|\eta| \le 10$. We split the integral $K(S,S)$ at $|\nu|=20$. For $|\nu| \le 20$, everything is smooth, so that as 
\[ \eta \mapsto \int_{|\nu| \le 20} e^{i \eta \nu^2} \tilde S(\eta,\nu) d\nu \in  C^\infty. \]
In particular,
\[ \int_{|\nu| \le 20} e^{\frac{3i}{4} \eta \nu^2} \tilde S(\eta,\nu) d\nu = \int_{|\nu| \le 20} S(\nu/2)^2 d\nu + O(|\eta|). \]
For $|\nu|>20$,
\[
\tilde S(\eta,\nu)= e^{ia\ln\left|\frac{\eta+\nu}{\eta-\nu}\right|}\left(A + 8B e^{2ia\ln|(\eta+\nu/2|}\frac{e^{i\beta (\eta+\nu)^3/8 }}{(\eta+\nu)^3}\right)\left(\overline{A} - 8\overline{B}e^{-2ia\ln|(\eta-\nu)/2|}\frac{e^{i\beta (\eta-\nu)^3/8 }}{(\eta-\nu)^3}\right)
\]
Since
\begin{align}\label{desenvolvimentoexp}
e^{ia\ln|\frac{\eta+\nu}{\eta-\nu}|}&=e^{ia\ln|1+\eta/\nu|}e^{-ia\ln|1-\eta/\nu|} = \left(1 + ia\frac{\eta}{\nu} + O(\eta^2/\nu^2)\right)\left(1 + ia\frac{\eta}{\nu} + O(\eta^2/\nu^2)\right) \nonumber\\&= 1+2ia\frac{\eta}{\nu} + O(\eta^2/\nu^2),
\end{align}
one develops the pure term in $A$ as
\begin{align*}
& \int_{|\nu|>20} e^{3i\eta\nu^2/4 + ia\ln|\frac{\eta+\nu}{\eta-\nu}|} d\nu = \frac{1}{\sqrt{|\eta|}} \int_{|\mu|>20\sqrt{|\eta|}} e^{3i\mu^2/4}\left(1+2ia\frac{|\eta|^{1/2}\eta}{\mu} + O(|\eta|^{5/2}/\mu^2)\right) d\mu \\& = \frac{1}{\sqrt{|\eta|}}\int_{|\mu|>20\sqrt{|\eta|}} e^{3i\mu^2/4} d\mu +  O(|\eta|) = \frac{1}{\sqrt{|\eta|}}\int e^{3i\mu^2/4} d\mu + \frac{1}{\sqrt{|\eta|}}\int_{|\mu|<20\sqrt{|\eta|}} e^{3i\mu^2/4} d\mu +  O(|\eta|) \\&= \frac{1}{\sqrt{|\eta|}}\int e^{3i\mu^2/4} d\mu + \int_{|\nu|<20} e^{3i\eta\nu^2/4} d\nu + O(|\eta|) = \frac{1}{\sqrt{|\eta|}}\int e^{3i\mu^2/4} d\mu + 40 + O(|\eta|).
\end{align*}
Now we exemplify how to estimate the remaining terms: we write
\[
\int_{|\nu|>20} e^{3i\eta\nu^2/4}e^{ia\ln|\frac{\eta+\nu}{\eta-\nu}|}e^{2ia\ln|(\eta+\nu)/2|}\frac{e^{i\beta(\eta+\nu)^3/8}}{(\eta+\nu)^3} d\nu = \int_{|\nu|>20} e^{i\Theta(\eta,\nu)}e^{2ia\ln|(\eta+\nu)/2|}\frac{e^{ia\ln|\frac{\eta+\nu}{\eta-\nu}|}}{(\eta+\nu)^3} d\nu
\]
where $\Theta(\eta,\nu)=3\eta\nu^2/4 + \beta(\eta+\nu)^3/8$. Notice that, since $|\nu|> 20$ and $|\eta|<10$,
\[
\partial_\nu \Theta \sim \nu^2,\quad \partial_{\nu\nu}^2 \Theta \sim \nu,\quad \partial_\eta \Theta \sim \nu^2,\quad \partial^2_{\nu\eta}\Theta \sim \nu. 
\]
Setting
\[
m_1(\eta,\nu):=e^{2ia\ln|(\eta+\nu)/2|}\frac{e^{ia\ln|\frac{\eta+\nu}{\eta-\nu}|}}{(\eta+\nu)^3},
\]
by integration by parts,
\begin{align*}
\MoveEqLeft \int_{|\nu|>20} e^{i\Theta(\eta,\nu)}m_1(\eta,\nu) d\nu \\
&= \left[\frac{e^{i\Theta(\eta,\nu)}}{i\partial_\nu \Theta(\eta,\nu)}m_1(\eta,\nu)\right]_{-20}^{20} + \int_{|\nu|>20} \frac{e^{i\Theta(\eta,\nu)}}{i\partial_\nu \Theta(\eta,\nu)}\left( \frac{\partial^2_{\nu\nu}\Theta(\eta,\nu)}{\partial_\nu\Theta(\eta,\nu)}m_1(\eta,\nu) + \partial_\nu m_1(\eta,\nu)\right)d\nu \\
& = \left[\frac{e^{i\Theta(\eta,\nu)}}{i\partial_\nu \Theta(\eta,\nu)}m_1(\eta,\nu)\right]_{-20}^{20} + \int_{|\nu|>20} \frac{e^{i\Theta(0,\nu)}}{i\partial_\nu \Theta(0,\nu)}\left( \frac{\partial^2_{\nu\nu}\Theta(0,\nu)}{\partial_\nu\Theta(0,\nu)}\frac{e^{2ia\ln|\nu/2|}}{\nu^3} + \partial_\nu m_1(0,\nu)\right)d\nu \\
& \quad + \int_{|\nu|>20}\int_0^\eta\partial_\zeta\left[ \frac{e^{i\Theta(\zeta,\nu)}}{i\partial_\nu \Theta(\zeta,\nu)}\left( \frac{\partial^2_{\nu\nu}\Theta(\zeta,\nu)}{\partial_\nu\Theta(\zeta,\nu)}m_1(\zeta,\nu)+ \partial_\nu m_1(\zeta,\nu) \right)\right]d\zeta d\nu \\
&  = \mbox{constant} + O(|\eta|),
\end{align*}

This concludes the proof of \eqref{est:l7_1}, and completes the estimates for $K(S,S)$.

\bigskip

We now turn to the estimates of $\partial_\eta K(S,S)$. For $\eta>0$, writing
\[
K(S,S)(\eta)=\frac{1}{\sqrt{|\eta|}}\int e^{3i\mu^2/4}\tilde{S}(\eta, \mu/\sqrt{|\eta|})d\mu,
\]
one computes
\begin{align*}
\partial_\eta K(S,S)(\eta) &= -\frac{1}{2\eta}K(S,S)(\eta) + \frac{1}{\sqrt{|\eta|}}\int e^{3i\mu^2/4}\left(\partial_\eta \tilde{S}(\eta,\mu/\sqrt{|\eta|}) -\partial_\nu\tilde{S}(\eta,\mu/\sqrt{|\eta|})\frac{\mu}{|\eta|^{1/2}\eta} \right)d\mu\\&= -\frac{1}{2\eta}K(S,S)(\eta) + \frac{1}{\eta}\int e^{3i\eta\nu^2/4}\left(\eta-\frac{\nu}{2}\right)S'\left(\frac{\eta+\nu}{2}\right)S\left(\frac{\eta-\nu}{2}\right)d\nu
\end{align*}
 Denote 
\[ \breve S(\eta,\nu) = \left( \eta - \frac{\nu}{2} \right) S' \left( \frac{\eta+\nu}{2} \right) S \left( \frac{\eta-\nu}{2} \right). \]
The claimed estimates for $\partial_\eta K(S,S)$ will follow applying to the second integral computations similar to those made for $K(S,S)$.

\bigskip

\emph{Proof of \eqref{est:l7_4}.} We split the integration at $|\nu|=20$. If $|\nu|<20$, the integrand is smooth and so 
\[
\int e^{3i\eta\nu^2/4}\left(\eta-\frac{\nu}{2}\right)S'\left(\frac{\eta+\nu}{2}\right)S\left(\frac{\eta-\nu}{2}\right)d\nu = O(1).
\]
In the region $|\nu|>20$, we write
\begin{gather*}
\int_{|\nu|>20} e^{3i\eta\nu^2/4}\left(\eta-\frac{\nu}{2}\right)S'\left(\frac{\eta+\nu}{2}\right)S\left(\frac{\eta-\nu}{2}\right)d\nu = \sum_{i,j=1,2} T_{3,ij} \\
\text{where} \quad T_{3,ij} := \int_{|\nu|>20} e^{3i\eta\nu^2/4}\left(\eta-\frac{\nu}{2}\right)S_i'\left(\frac{\eta+\nu}{2}\right)S_j\left(\frac{\eta-\nu}{2}\right)d\nu.
\end{gather*}
The terms $T_{3,12}$ and $T_{3,22}$ are brutally estimated:
\[
|T_{3,12}| + |T_{3,22}| \lesssim \int_{|\nu|>20} |\nu|^{-3}d\nu = O(1).
\]
For $T_{3,11}$, using \eqref{desenvolvimentoexp},
\begin{align*}
\MoveEqLeft \int_{|\nu|>20} e^{3i\eta\nu^2/4}\left(\eta-\frac{\nu}{2}\right)\frac{2ia}{\eta + \nu}e^{ia\ln|\frac{\eta+\nu}{\eta-\nu}|}d\nu \\
&= \frac{1}{\sqrt{|\eta|}}\int_{|\mu|>20\sqrt{|\eta|}} e^{3i\mu^2/4} \left(\frac{|\eta|^{1/2}\eta/\mu - 1/2}{1 + |\eta|^{1/2}\eta/\mu}\right)\left(1+2ia\frac{|\eta|^{1/2}\eta}{\mu} + O(|\eta|^3/|\mu|^2)\right)d\mu \\ 
& = \frac{1}{\sqrt{|\eta|}}\int_{|\mu|>20\sqrt{|\eta|}} e^{3i\mu^2/4}\left(-\frac{1}{2} + \left(\frac{3}{2}-ia\right)\frac{|\eta|^{1/2}\eta}{\mu} + O(|\eta|^3/|\mu|^2)\right)d\mu \\ & = -\frac{1}{2\sqrt{|\eta|}}\int e^{3i\mu^2/4} d\mu + \frac{1}{2\sqrt{|\eta|}}\int_{|\mu|<20\sqrt{|\eta|}} e^{3i\mu^2/4} d\mu + O(1) \\
& = -\frac{1}{2\sqrt{|\eta|}}\int e^{3i\mu^2/4} d\mu + O(1).
\end{align*}
For $T_{3,21}$,
\begin{align*}
\MoveEqLeft \int_{|\nu|>20} e^{3i\eta\nu^2/4}\left(\eta-\frac{\nu}{2}\right)\frac{3i\beta}{\eta + \nu}e^{2ia\ln|(\eta+\nu)/2|}e^{i\beta(\eta+\nu)^3/8}e^{ia\ln|\frac{\eta+\nu}{\eta-\nu}|}d\nu \\
&= \frac{3i\beta}{2} \int_{|\nu|>20} e^{i\Theta(\eta,\nu)} \frac{2\eta-\nu}{\eta+\nu}e^{2ia\ln|(\eta+\nu)/2|}e^{ia\ln|\frac{\eta+\nu}{\eta-\nu}|}d\nu\\&= \frac{3i\beta}{2} \int_{|\nu|>20}e^{i\Theta(\eta,\nu)} \frac{-1 + 2\eta/\nu}{1+\eta/\nu}e^{2ia\ln|\nu/2|}\left(1+2ia\frac{\eta}{\nu}\right)\left(1+2ia\frac{\eta}{\nu} + O(|\eta|^2/|\nu|^2) \right) d\nu \\&= \frac{3i\beta}{2} \int_{|\nu|>20}e^{i\Theta(\eta,\nu)} e^{2ia\ln|\nu/2|}\left( -1 + (4-2ia)\frac{\eta}{\nu} +  O(|\eta|^2/|\nu|^2)\right)d\nu \\
&= \frac{3i\beta}{2} \int_{|\nu|>20}e^{i\Theta(\eta,\nu)} e^{2ia\ln|\nu/2|}\left( -1 + (4-2ia)\frac{\eta}{\nu}\right)d\nu + O(1)
\end{align*}
The first term is bounded:
\begin{align*}
\MoveEqLeft \int_{|\nu|>20}e^{i\Theta(\eta,\nu)}e^{2ia\ln|\nu/2|}d\nu \\
&= -\left[ \frac{e^{i\Theta(\eta,\nu)}e^{2ia\ln|\nu/2|}}{i\partial_\nu\Theta(\eta,\nu)}\right]_{-20}^{20} + \int_{|\nu|>20}\frac{e^{i\Theta(\eta,\nu)}}{i\partial_\nu \Theta(\eta,\nu)}\left(\frac{\partial_{\nu\nu}^2\Theta(\eta,\nu)}{\partial_\nu \Theta (\eta,\nu)}e^{2ia\ln|\nu/2|} + \frac{2ia e^{2ia\ln|\nu/2|}}{\nu}\right) d\nu \\
&= O(1) - \left[\frac{e^{i\Theta(\eta,\nu)}}{(\partial_\nu \Theta(\eta,\nu))^2}\left(\frac{\partial_{\nu\nu}^2\Theta(\eta,\nu)}{\partial_\nu \Theta (\eta,\nu)}e^{2ia\ln|\nu/2|} + \frac{2ia e^{2ia\ln|\nu/2|}}{\nu}\right)\right]_{-20}^{20} \\
&\quad - \int_{|\nu|>20} e^{i\Theta(\eta,\nu)}\partial_\nu\left(\frac{1}{(\partial_\nu \Theta(\eta,\nu))^2}\left(\frac{\partial_{\nu\nu}^2\Theta(\eta,\nu)}{\partial_\nu \Theta (\eta,\nu)}e^{2ia\ln|\nu/2|} + \frac{2ia e^{2ia\ln|\nu/2|}}{\nu}\right)\right) \\
& = O(1)
\end{align*}
where we used the fact that, over this region,
\[
\partial_\nu \Theta \sim \nu^2,\quad \partial_{\nu\nu}^2 \Theta \sim \nu,\quad \partial^3_{\nu\nu\nu} \Theta \sim 1.
\]
The second term is also handled with an integration by parts:
\begin{align*}
\MoveEqLeft \int_{|\nu|>20}e^{i\Theta(\eta,\nu)}e^{2ia\ln|\nu/2|}\frac{\eta}{\nu}d\nu = -\left[\frac{e^{i\Theta(\eta,\nu)}e^{2ia\ln|\nu/2|}\eta}{i\nu \partial_{\nu}\Theta(\eta,\nu)} \right]_{-20}
^{20} \\
&\qquad- \int_{|\nu|>20}e^{i\Theta}\left(\frac{\partial_{\nu\nu}^2\Theta}{(\partial_\nu\Theta)^2}e^{2ia\ln|\nu/2|}\frac{\eta}{\nu} - \frac{1}{\partial_{\nu}\Theta}e^{2ia\ln|\nu/2|}\left(\frac{\eta}{\nu^2} + \frac{2ia\eta}{\nu^2}\right)\right)d\nu \\
&= O(\eta)  + O\left(\int_{|\nu|>20}\frac{\eta}{\nu^3}d\nu\right) = O(\eta)
\end{align*}

\emph{Proof of \eqref{est:l7_5}.} Now we consider the case $\eta>10$. We write
\begin{align*}
\MoveEqLeft \int e^{3i\eta\nu^2/4}\left(\eta-\frac{\nu}{2}\right)S'\left(\frac{\eta+\nu}{2}\right)S\left(\frac{\eta-\nu}{2}\right)d\nu= T_{4,1} + T_{4,2} + T_{4,3} \quad  \text{where}\\
T_{4,1} & =  \int_{|\nu|<|\eta|/2} e^{3i\eta\nu^2/4}\left(\eta-\frac{\nu}{2}\right)S_1'\left(\frac{\eta+\nu}{2}\right)S\left(\frac{\eta-\nu}{2}\right)d\nu  \\
T_{4,2} & = \int_{|\nu|>|\eta|/2} e^{3i\eta\nu^2/4}\left(\eta-\frac{\nu}{2}\right)S_1'\left(\frac{\eta+\nu}{2}\right)S\left(\frac{\eta-\nu}{2}\right)d\nu \\
T_{4,3} & = \int e^{3i\eta\nu^2/4}\left(\eta-\frac{\nu}{2}\right)S_2'\left(\frac{\eta+\nu}{2}\right)S\left(\frac{\eta-\nu}{2}\right)d\nu.
\end{align*}
For $T_{4,1}$, we have, for some function $\phi$ with $\|\phi\|_{W^{1,\infty}(-1/2,1/2)}<\infty$,
\begin{align*}
T_{4,1} &= A\int_{|\nu|<|\eta|/2} e^{3i\eta\nu^2/4}\frac{2\eta-\nu}{\eta+\nu}e^{ia\ln|(\eta^2-\nu^2)/4|}\left( A + 8Be^{-2ia\ln|(\eta-\nu)/2|}\frac{e^{i\beta(\eta-\nu)^3/8}}{(\eta-\nu)^3}\right)d\nu \\&= \frac{A}{2} \int_{|\nu|<|\eta|/2} e^{3i\eta\nu^2/4}\frac{2\eta-\nu}{\eta+\nu}e^{ia\ln|(\eta^2-\nu^2)/4|} + O(|\eta|^{-2})\\&= A^2|\eta|e^{ia\ln(\eta^2/4)}\int_{|z|<1/2}e^{3i\eta^3z^2/4}\frac{2-z}{1+z}e^{ia\ln(1-z^2)}dz + O(|\eta|^{-2})\\&= A^2|\eta|e^{ia\ln(\eta^2/4)}\int_{|z|<1/2}e^{3i\eta^3z^2/4}\left(1+z\phi(z)\right)dz.
\end{align*}
We then proceed as in the non-derivative case and obtain
\[
T_{4,1}=iaA^2\sqrt{\frac{4\pi}{3|\eta|}}e^{i\pi/4}e^{ia\ln(\eta^2/4)} + O(|\eta|^{-2}).
\]
The term $T_{4,2}$ is handled as $T_{2,2}$, using integration by parts and the generic bounds on $S$:
\[
T_{4,2}=O(|\eta|^{-2}).
\]
 Finally, we consider $T_{4,3}=T_{4,4}+T_{4,5}$, where
\begin{align*}
T_{4,4} &= B\int e^{i\Theta(\eta,\nu)}\frac{\eta-2\nu}{\eta+\nu}\Bigg(e^{2ia\ln|(\eta+\nu)/2|}\left(1+\frac{2ia}{(\eta+\nu)^3}\right)\chi\left( \frac{\eta+\nu}{2}\right)\\&\qquad\qquad + e^{-2ia\ln|(\eta+\nu)/2|}\left(1-\frac{2ia}{(\eta+\nu)^3}\right)\chi\left( -\frac{\eta+\nu}{2}\right)\Bigg) S\left(\frac{\eta-\nu}{2}\right)d\nu
\end{align*}
and
\begin{align*}
T_{4,5} &=  8B\int e^{3i\eta\nu^2}\left(\eta-\frac{\nu}{2}\right)\frac{e^{i\beta(\eta+\nu)^3/8}}{(\eta+\nu)^3}\Bigg(e^{2ia\ln|(\eta+\nu)/2|}\chi'\left(\frac{\eta+\nu}{2}\right)\\&\qquad\qquad-e^{-2ia\ln|(\eta+\nu)/2|}\chi'\left(-\frac{\eta+\nu}{2}\right)\Bigg)S\left(\frac{\eta-\nu}{2}\right)d\nu.
\end{align*}
Notice that $\chi'$ has compact support. Hence the term $K_5$ can be handled by successive integrations by parts, using the fact that the integrand is $C^\infty_c(\real)$. We then focus on $K_4$. Define
\[
Q(X)=\frac{3}{4}X^2 + \frac{\beta}{8}(1+X)^3\mbox{ so that }\Theta(\eta,\nu)=\eta^3Q(\nu/\eta).
\]
We consider the case where the polynomial $Q'$ has two distinct zeros $r_\pm$ (that is, when $\beta>-1$). Set $Q_\pm=Q(X_\pm)$ and $Q''_\pm = Q''(X_\pm)$. For a fixed $\epsilon$ small, we take smooth cut-off functions $\theta_\pm$ such that
\[
\theta_\pm \equiv 1 \text{ on } (r_\pm-\epsilon, r_\pm+\epsilon),\quad  \theta_\pm \equiv 0 \text{ on } \real\setminus \left(r_\pm-2\epsilon, r_\pm+2\epsilon \right).
\]
Write
\begin{align*}
m_2(\eta,\nu)&=\frac{\eta-2\nu}{\eta+\nu}e^{2ia\ln|(\eta+\nu)/2|}\left(1+\frac{2ia}{(\eta+\nu)^3}\right)\chi\left( \frac{\eta+\nu}{2}\right)S\left(\frac{\eta-\nu}{2}\right) \\ &\qquad + \frac{\eta-2\nu}{\eta+\nu}e^{-2ia\ln|(\eta+\nu)/2|}\left(1-\frac{2ia}{(\eta+\nu)^3}\right)\chi\left( -\frac{\eta+\nu}{2}\right)S\left(\frac{\eta-\nu}{2}\right)
\end{align*}
and split $T_{4,4}$ as
\begin{align*}
T_{4,4} &= \sum_{j\in\{\pm\}}B\int e^{i\Theta(\eta,\nu)}m_2(\eta,\nu)\theta_j(\nu/\eta)d\nu \\&+ B\int e^{i\Theta(\eta,\nu)}m_2(\eta,\nu)(1-\theta_+-\theta_-)(\nu/\eta)d\nu.
\end{align*}
Writing $m_3(\eta,\nu)=m_2(\eta,\nu)(1-\theta_+-\theta_-)(\nu/\eta)$, some direct computations yield
\[
m_3(\eta,\nu)=O(1),\quad \partial_\nu m_3(\eta,\nu)=O(|\nu|^{-1}),\quad  \partial_{\nu\nu}^2 m_3(\eta,\nu)=O(|\nu|^{-2}),\quad |\nu|>2|\eta|
\]
and, because of the term $(\eta-2\nu)/(\eta+\nu)$ for $\eta+\nu$ close to 1,
\[
m_3(\eta,\nu)=O(|\eta|),\quad \partial_\nu m_3(\eta,\nu)=O(|\eta|),\quad  \partial_{\nu\nu}^2 m_3(\eta,\nu)=O(|\eta|),\quad |\nu|<2|\eta|
\]
We treat the last integral with a stationary phase argument:
\begin{align*}
\MoveEqLeft \left|\int e^{i\Theta(\eta,\nu)}m_3(\eta,\nu)d\nu\right| = \left|\int \frac{e^{i\Theta(\eta,\nu)}}{i\partial_\nu \Theta(\eta,\nu)}\left(\frac{\partial_{\nu\nu}^2\Theta(\eta,\nu)}{\partial_{\nu}\Theta(\eta,\nu)}m_3(\eta,\nu) + \partial_\nu m_3(\eta,\nu)\right)d\nu\right|\\ & = \left|\int e^{i\Theta}\partial_\nu \left(  -\frac{\partial_{\nu\nu}^2\Theta(\eta,\nu)}{(\partial_{\nu}\Theta(\eta,\nu))^3}m_3(\eta,\nu) - \frac{\partial_\nu m_3(\eta,\nu)}{(\partial_{\nu}\Theta(\eta,\nu))^2}\right)d\nu\right| \\ & \lesssim \int_{|\nu|>2|\eta|} \frac{1}{\nu^6}d\nu + \int_{|\nu|<2|\eta|} \frac{1}{\eta^3}d\nu = O(|\eta|^{-2}).
\end{align*}

For the stationary regions (for example, close to $r_-$), one considers a bijection $\lambda_-:I\to [r_--2\epsilon, r_-+2\epsilon]$ such that
\[
Q(\lambda_-(\mu))=Q_- + \frac{Q''_-}{2}\mu^2,\quad  \lambda_-(0)=r_-,\quad \lambda_-'(0)=1,\quad |\lambda_-'|>1/2.
\]
and the change of variables $\nu=\lambda_-(\mu/\eta)\eta$:
\begin{align*}
\MoveEqLeft \int e^{i\Theta(\eta,\nu)}\frac{\eta-2\nu}{\eta+\nu}S\left(\frac{\eta-\nu}{2}\right)m_2(\eta,\nu)\theta_-(\nu/\eta)d\nu \\&= e^{i\eta^3Q_-}\int e^{iQ''_-\eta\mu^2/2} S\left(\frac{\eta(1-\lambda_-(\mu/\eta))}{2}\right)\frac{1-2\lambda_-(\mu/\eta)}{1+\lambda_-(\mu/\eta)}m_2(\eta,\lambda_-(\mu/\eta)\eta)\theta_-(\lambda_-(\mu/\eta))\lambda_-'(\mu/\eta)d\mu\\&=:e^{i\eta^3Q_-}\int e^{iQ''_-\eta\mu^2/2} h(\mu,\eta)d\mu
\end{align*}
Since
\[
|h(\mu,\eta)|<1,\quad |\partial_\eta h(\mu,\eta)|\le \frac{1}{|\eta|},
\]
the application of \eqref{eq:fond1} yields
\begin{align*}
\MoveEqLeft e^{i\eta^3Q_-}\int e^{iQ''_-\eta\mu^2/2} h(\mu,\eta)d\mu\\&= e^{iQ_-\eta^3}\sqrt{\frac{2\pi}{|Q''_-||\eta|}}e^{i\frac{\pi}{4}\sgn(Q_-''\eta)}m_2(\eta,r_-\eta)+ O(|\eta|^{-2}\ln|\eta|)
\end{align*}
Therefore, for $\eta$ large,
\begin{align*}
T_{4,4}  &= \sum_{j\in\{\pm\}}B\int e^{i\Theta(\eta,\nu)}m_2(\eta,\nu)\theta_j(\nu/\eta)d\nu + O(|\eta|^{-2}\ln|\eta|)\\
&= e^{iQ_-\eta^3}\sqrt{\frac{2\pi}{|Q''_-||\eta|}}e^{i\frac{\pi}{4}\sgn(Q_-''\eta)}\frac{1-2r_-}{1+r_-}\left(e^{2ia\ln|(1+r_-)\eta/2|} + e^{-2ia\ln|(1+r_-)\eta/2|}\right)Ae^{ia\ln|(1+r_-)\eta/2|} \\ 
& \quad  +e^{iQ_+\eta^3}\sqrt{\frac{2\pi}{|Q''_+||\eta|}}e^{i\frac{\pi}{4}\sgn(Q_+''\eta)}\frac{1-2r_+}{1+r_+}\left(e^{2ia\ln|(1+r_+)\eta/2|} + e^{-2ia\ln|(1+r_+)\eta/2|}\right)Ae^{ia\ln|(1+r_+)\eta/2|} \\
& \quad + O(|\eta|^{-2}\ln|\eta|)\\
& =: \frac{A_+(\eta)}{\sqrt{|\eta|}} + O(|\eta|^{-2}\ln|\eta|)
\end{align*}
where $A_+$ is a bounded function. Gathering the estimates for $T_{4,1}, T_{4,2}$ and $T_{4,3}$, one arrives at the claimed result. 

This concludes the proof of estimate \eqref{est:l7_5}. The proof of \eqref{est:l7_6} is completely analogous.
\end{proof}

\begin{lem}\label{lem:desenvolveISSS}
Fix $6/7<\gamma<1$. Then for $|\xi| \le 10$,
\begin{align} \label{est:l8_1}
I(S,S,S)(\xi)=O(1).
\end{align}
Also, for $\xi \ge 10$,
\begin{align} \label{est:l8_2}
I(S,S,S)(\xi)  = \frac{e^{ia \ln|\xi|}}{|\xi|}\left(E + F e^{2ia \ln|\xi|}e^{-8i\xi^3/9}\right) + O(|\xi|^{-2+\gamma/2})
\end{align}
and for $\xi \le -10$,
\begin{align} \label{est:l8_3}
I(S,S,S)(\xi)  = \frac{e^{-ia \ln|\xi|}}{|\xi|}\left( \bar E + \bar F e^{-2ia \ln|\xi|}e^{8i\xi^3/9}\right) + O(|\xi|^{-2+\gamma/2})
\end{align}
where the constants $E,F \in \m C$ are defined by
\begin{gather} \label{est:l8_4}
E = \pi|A|^2A, \quad \text{and} \quad F = i \frac{\sqrt{2}\pi}{3} e^{ia\ln 3}|A|^2A.
\end{gather}
\end{lem}
\begin{proof}
Set
\[
f(\xi)=K(S,S)(\xi), \quad g(\xi)=\bar S(\xi).
\]
Then, using Lemma \ref{lem:desenvolveISS}, we have
\[
|f(\xi)| = O(|\xi|^{-1/2}),\quad |f'(\xi)| = O(|\xi|^{-3/2}).
\]
It then follows from Lemma \ref{lem:primeiro} that for $|\xi| < 10$,
\[
I(S,S,S)(\xi)=O(1),
\]
which proves \eqref{est:l8_1}. 

\bigskip

The estimate \eqref{est:l8_3} can be derived from  \eqref{est:l8_2} by symmetry as
\[ I(S,S,S)(\xi) = \overline{I(S,S,S)(-\xi)}. \]
Hence it suffices to prove the latter.

\emph{Proof of \eqref{est:l8_2}.} If $\xi \ge 10$, we have from Lemmas \ref{lem:nonstat} and \ref{lem:stat},
\begin{align*}
I(S,S,S)(\xi)  &= \sqrt{\frac{\pi}{3|\xi|}}\left(e^{-i\sgn(\xi)\pi/4}f(2\xi)g(\xi)\frac{\sqrt{2}}{2} + e^{i \sgn(\xi)\pi/4}f(2\xi/3)g(-\xi/3)e^{-8i\xi^3/9}\frac{1}{\sqrt{3}}\right) \\
& \quad + \frac{1}{2}J_1(K(S,S), \bar S) + O(|\xi|^{-5/2}\ln|\xi|).
\end{align*}
We focus on the term $J_1(K(S,S),\bar S)$, which cannot be estimated using Lemma \ref{lem:regsing}.

If $\eta/\xi$ is in the support of $\varphi_1$, $|\eta| \le 3/8 \cdot 7/6 |\xi| = 7/16 |\xi|$. As $|\xi| \ge 10$, $|\eta - \xi| \ge 9|\xi|/16 \ge 5$ and 
\[
S(\eta - \xi) = e^{ia \ln|\eta-\xi|} \left(A + B \frac{e^{i\beta (\eta-\xi)^3}}{(\eta-\xi)^3}\right).
\]

As $P(0)=0$ and $P'$ does not vanish on $(-1/2, 1/2)$ ($P'(0)=1$), there exists a diffeomorphism $\psi_3: \left(-c_3,d_3 \right) \to \left(-1/2,1/2 \right)$ to its image, and such that
\[ \forall \nu \in  \left(-c_3,d_3 \right), \quad  P(\psi_3(\nu)) = \nu. \]
($\psi'(0)=1$, $c_3 = -P(-1/2)<3/4$, $d_3 = P(1/2)<1/2$). We extend $\psi_3$ to a diffeomorphism $\real \to \real$ such that for all $|\nu| \ge 10$,  $\psi_3(\nu)=\nu$. In particular, for some constant $C_3>0$,
\[ \forall \nu \in \real, \quad 0 < 1/C_3 \le \psi_3'(\nu) \le C_3. \]
 Also let $C_3$ be such that
\[ \forall \nu \in [-10,10], \quad | \psi_3(\nu) - \nu | \le C_3 \nu^2, \quad | \psi_3'(\nu) -1 | \le C_3 \nu. \]
Hence for all $|\eta| \le \xi/2$, there holds, with $\eta = \psi_3(\mu/\xi) \xi)$
\[ \Phi(\xi, \eta) = \xi^3 P (\mu/\xi) = \xi^2 \mu. \]

We now decompose in three terms:
\begin{align*}
\int  e^{-3i \Phi(\xi, \eta)} & K(S,S)(\eta) \bar S(\eta - \xi) \varphi_1(\eta/\xi) d\eta 
 = T_{5,1} + T_{5,2} + T_{5,3} \qquad  \text{with} \\
T_{5,1} & =  \int_{\eta \ge 0} e^{-3i \Phi(\xi, \eta)} K(S,S)(\eta) \bar S(\eta - \xi) \phi(|\xi|^{\gamma} \eta) d\eta \\
T_{5,2} & =  \int_{\eta \le 0} e^{-3i \Phi(\xi, \eta)} K(S,S)(\eta) \bar S(\eta - \xi) \phi(|\xi|^{\gamma} \eta) d\eta \\
T_{5,3} & =  \int e^{-3i \Phi(\xi, \eta)} K(S,S)(\eta) \bar S(\eta - \xi) (\varphi_1(\eta/\xi)  -  \phi(|\xi|^{\gamma} \eta)) d\eta
\end{align*}
Then, for $\xi>0$,
\begin{align*}
T_{5,1} & = e^{i \pi /4} \sqrt{\frac{4\pi}{3}}  |A|^2 \int_{\eta \ge 0} e^{-3i \Phi(\xi, \eta)} \frac{1}{\sqrt{|\eta|}} e^{ia\ln|\eta-\xi|}\left(A+ Be^{2ia\ln|(\eta-\xi)/2|}\frac{e^{-i\beta(\eta-\xi)^3}}{(\eta-\xi)^3}\right) \phi(|\xi|^{\gamma} \eta) d\eta \\& \quad+ D\int_0^{2/|\xi|^{\gamma}} e^{-3i\Phi(\xi,\eta)}\phi(|\xi|^{\gamma} \eta) d\eta + O\left(\int_0^{2|\xi|^{-\gamma}}|\eta|d\eta\right) 
\end{align*}
The last term gives $O(|\xi|^{-2\gamma})$. The penultimate term:
\begin{align*}
\int_0^{2/|\xi|^{\gamma}} e^{-3i\Phi(\xi,\eta)}\phi(|\xi|^{\gamma} \eta) d\eta & = \int_0^{2/|\xi|^{\gamma}} e^{-3i\eta\xi^2} \phi(|\xi|^{\gamma} \eta) + e^{-3i\eta\xi^2}\left(e^{3i\xi\eta^2 - 3i\eta^3/4}-1\right)\phi(|\xi|^{\gamma} \eta) d\eta\\ & = O(|\xi|^{-2}) + O\left(\int_0^{|\xi|^{-\gamma}} |\eta|^2|\xi|d\eta\right) = O(|\xi|^{1-3\gamma})
\end{align*}
The second term is brutally bounded by
\[
\int_{0}^{|\xi|^{-\gamma}} |\xi|^{-3}|\eta|^{-1/2} d\eta = O(|\xi|^{-3+\gamma/2})
\]
For the first term,
\begin{align*}
\MoveEqLeft \int_{\eta \ge 0} e^{-3i \Phi(\xi, \eta)} \frac{e^{-i a \ln |\eta|} }{\sqrt{|\eta|}} e^{i a \ln (\xi - \eta)} \phi(|\xi|^{\gamma} \eta) d\eta \\
& = e^{i a \ln |\xi|} \int_{\eta \ge 0} e^{-3i \Phi(\xi, \eta)}\frac{e^{-i a \ln |\eta|}}{\sqrt{|\eta|}} (1 + O(\eta/\xi)) \phi(|\xi|^{\gamma} \eta) d\eta \\
& = e^{i a \ln |\xi|} \int_{\eta \ge 0} e^{-3i \Phi(\xi, \eta)}\frac{e^{-i a \ln |\eta|}}{\sqrt{|\eta|}} \phi(|\xi|^{\gamma} \eta) d\eta + O(|\xi|^{-1-3\gamma/2}) 
\end{align*}

Performing the change of variables $\eta=\phi_3(\eta/\xi^3)\xi$,
\begin{align*}
\MoveEqLeft \int_{\eta \ge 0} e^{-3i \Phi(\xi, \eta)}\frac{1}{\sqrt{|\eta|}} \phi(|\xi|^{\gamma} \eta) d\eta =  \int_{\mu \ge 0} e^{3 i \xi^2 \mu} \frac{1 }{\sqrt{|\psi_3(\mu/\xi)\xi |}} \phi( |\xi|^{\gamma} \xi \psi_3(\mu/\xi))  \psi_3'(\mu/\xi) d\mu \\
& = \frac{1}{|\xi|} \int_{\nu \ge 0} e^{-3i \nu} \frac{1}{\sqrt{|\psi_3(\nu/\xi^3)\xi^3|}} \phi( |\xi|^{1+\gamma} \psi_3(\nu/\xi^3))  \psi_3'(\nu/\xi^3) d\nu
\end{align*}
Now the integrated term vanishes as soon as $|\xi|^{1+\gamma} \psi_3(\nu /\xi^3) \ge 7/6$. But if $|\nu| \ge 7K/6 \cdot |\xi|^{2-\gamma}$, $\psi_3(\nu /\xi^3) \ge 7 |\xi|^{-1-\gamma}/6 $ and $\psi_3(\nu /\xi^3) |\xi|^3) =0$. Hence we can assume $\nu \in [0,7|\xi|^{2-\gamma}/6]$, so that $|\nu / \xi^3| \le 2 | \xi|^{-1-\gamma} \le 1/10$. Thus
\[ \left| \frac{1}{\sqrt{|\psi_3(\nu/\xi^3)\xi^3|}} - \frac{1}{\sqrt{|\nu|}} \right| \le K \frac{\sqrt{|\nu|}}{|\xi|^3} . \]
and so
\begin{align*}
\MoveEqLeft \int_{\eta \ge 0} e^{-3i \Phi(\xi, \eta)}\frac{1}{\sqrt{|\eta|}} \phi(|\xi|^{\gamma} \eta) d\eta =   \frac{1}{|\xi|} \int_{\nu \ge 0} e^{-3i \nu} \frac{1}{\sqrt{|\nu|}} \phi( |\xi|^{1+\gamma} \psi_3(\nu/\xi^3))  \psi_3'(\nu/\xi^3) d\nu \\
& \qquad + O \left( \frac{1}{|\xi|}\int_0^{2K|\xi|^{2-\gamma}}  \frac{\sqrt{|\nu|} d\nu}{|\xi|^3} \right) \\
& =   \frac{1}{|\xi|} \int_{\nu \ge 0} e^{-3i \nu} \frac{1}{\sqrt{|\nu|}} d\nu \\
& \qquad +  \frac{1}{|\xi|} \int_{\nu \ge 0} e^{-3i \nu} \frac{1}{\sqrt{|\nu|}} \left( \phi( |\xi|^{1+\gamma} \psi_3(\nu/\xi^3))  \psi_3'(\nu/\xi^3) -1 \right) d\nu + O(|\xi|^{-1-3\gamma/2})
\end{align*}
Observe that
\begin{align*}
\phi( |\xi|^{1+\gamma} \psi_3(\nu/\xi^3))  \psi_3'(\nu/\xi^3) -1 & =  \phi( |\xi|^{1+\gamma} \psi_3(\nu/\xi^3))-1 +  \phi( |\xi|^{1+\gamma} \psi_3(\nu/\xi^3))(  \psi_3'(\nu/\xi^3) -1) \\
& = O \left( \m 1_{|\nu| \sim |\xi|^{2-\gamma}} \right) + O \left( \nu/\xi^3 \m 1_{|\nu| \le |\xi|^{2-\gamma}} \right), \\
\partial_\nu \left( \phi( |\xi|^{1+\gamma} \psi_3(\nu/\xi^3))  \psi_3'(\nu/\xi^3) -1 \right) & = O \left( |\xi|^{-(2-\gamma)}  \m 1_{|\nu| \sim |\xi|^{2-\gamma}} \right) + O \left( |\xi|^{-3} \m 1_{|\nu| \le |\xi|^{2-\gamma}} \right).
\end{align*}
Hence, with the phase $e^{3i \nu}$,
\begin{align*}
\MoveEqLeft \frac{1}{|\xi|}\int_{\nu \ge 0} e^{-3i \nu} \frac{e^{-i a \ln |\nu|}}{\sqrt{|\nu|}} \left( \phi( |\xi|^{1+\gamma} \psi_3(\nu/\xi^3))  \psi_3'(\nu/\xi^3) -1 \right) d\nu \\
& =  \frac{1}{|\xi|}\Bigg(\int_{|\xi|^{2-\gamma}/10}^{10 |\xi|^{2-\gamma}} O \left( |\nu|^{-3/2} \right) d\nu + \int_0^{10 |\xi|^{2-\gamma}} O \left( |\nu|^{-1/2} /|\xi|^3 \right) d\nu \\
& \qquad  + \int_{|\xi|^{2-\gamma}/10}^{10 |\xi|^{2-\gamma}} O \left( |\nu|^{-1/2} |\xi|^{-(2-\gamma)} \right) d\nu + \int_0^{10|\xi|^{2-\gamma}} O \left(  |\xi|^{-3} |\nu|^{-1/2}\right) d\nu \Bigg)\\
& = O( |\xi|^{-2+\gamma/2} )
\end{align*}
where the main contribution comes from the first and third terms.  
Thus we arrive at
\begin{align*}
T_{5,1} & =e^{i \pi /4} \sqrt{\frac{4\pi}{3}} A |A|^2 \frac{e^{ia \ln|\xi|}}{|\xi|}\int_{\nu>0} e^{-3i\nu}\frac{1}{\sqrt{|\nu|}}d\nu + O(|\xi|^{-2\gamma}) + O(|\xi|^{1-3\gamma}) + O(|\xi|^{-2+\gamma/2})\\ & = \frac{2\pi}{3}  A |A|^2 \frac{e^{ia \ln|\xi|}}{|\xi|} + O(|\xi|^{-2+\gamma/2})
\end{align*}
(recall that $6/7<\gamma<1$). Analogously, one may prove that
\[
T_{5,2}=\frac{2\pi}{3}A |A|^2 \frac{e^{ia \ln|\xi|}}{|\xi|} + O(|\xi|^{-2+\gamma/2}).
\]
Finally, we look at $T_{5,3}$: performing an integration by parts, we have
\begin{align*}
 T_{5,3} & = \int e^{-3i \Phi(\xi,\eta)} \partial_\eta \left( \frac{1}{3\partial_{\eta} \Phi(\xi,\eta)} K(S,S)(\eta) \bar S(\eta-\xi) (\varphi_1(\eta/\xi) - \phi(\xi^{\gamma} \eta)) \right) d\eta \\
& = \int_{|\xi|^{-\gamma}}^{|\xi|/2} \frac{1}{\xi^2} O \left( \frac{1}{|\xi| |\eta|^{1/2}} \right) d\eta + \int_{|\xi|^{-\gamma}/2}^{2|\xi|^{-\gamma}} \frac{1}{\xi^2} O \left( \frac{|\xi|^{\gamma}}{|\eta|^{1/2}} \right) d\eta\\
& \qquad + \int e^{-3i \Phi(\xi,\eta)} \frac{1}{3\partial_{\eta} \Phi(\xi,\eta)} \partial_\eta K(S,S)(\eta) \bar S(\eta-\xi) (\varphi_1(\eta/\xi) - \phi(\xi^{\gamma} \eta))  d\eta.
\end{align*}
The first term gives $O(|\xi|^{-5/2})$, the second $O(|\xi|^{-2+\gamma/2})$. For the last term, one must use the asymptotics for $\partial_\eta K(S,S)$. Due to \eqref{est:l7_4}, \eqref{est:l7_5} and \eqref{est:l7_6},  
\[ \partial_{\eta} K(S,S)(\eta) = O(|\eta|^{-3/2}) \]
uniformly on $\eta \in \m R^*$, so that

\begin{align*}
\MoveEqLeft \int e^{-3i \Phi(\xi,\eta)} \frac{1}{3\partial_{\eta} \Phi(\xi,\eta)} \partial_\eta K(S,S)(\eta) \bar S(\eta-\xi) (\varphi_1(\eta/\xi) - \phi(\xi^{\gamma} \eta))  d\eta \\ 
& = \int_{\xi^{-\gamma}/2}^{2\xi}  O \left(\frac{1}{\xi^2|\eta|^{3/2}} \right)= O(|\xi|^{-2+\gamma/2}).
\end{align*}

The conclusion is that, for $\xi>2$,
\begin{align*}
I(S,S,S)(\xi) &= \sqrt{\frac{\pi}{3 \xi}}\left(e^{-i\frac{\pi}{4}}K(S,S)(2\xi)\overline{S}(\xi)\frac{\sqrt{2}}{2}+ e^{i\frac{\pi}{4}}K(S,S)(2\xi/3)\overline{S}(-\xi/3)e^{-8i\xi^3/9}\frac{1}{\sqrt{3}}\right) \\ 
& \qquad + \frac{2\pi}{3}A |A|^2\frac{e^{ia \ln(\xi)}}{|\xi|} + O(|\xi|^{-2+\gamma/2}) \\ 
& = \frac{e^{ia \ln|\xi|}}{|\xi|}\left(E + Fe^{2ia \ln|\xi|}e^{-8i\xi^3/9}\right) + O(|\xi|^{-2+\gamma/2})
\end{align*}
where $E$ and $F$ are given by \eqref{est:l8_4}.
\end{proof}

\section{Construction of a self-similar solution}

\subsection{Matching the asymptotics} \label{sec:5.1}

Using the computations of the previous section, we now adjust the constants $A,B,a,\beta,c,\alpha$ to obtain the final ansatz around which a fixed point argument is likely to run.

We recall that $E$ and $F$ are defined explicitly in $A$ in \eqref{est:l8_4}. Define for simplicity of notation 
\begin{align*} \label{def:tilde_I}
\tilde I(v) := I(v,v,v).
\end{align*}

As the integral of $\tilde I(S) - E \frac{e^{i a \ln |\xi|}}{|\xi|}$ is convergent on $[1,+\infty)$ (due to \eqref{est:l8_2}) and that
\[ \int_1^\xi \frac{e^{i a \ln |\eta|}}{|\eta|} d\eta = \frac{1}{ia} (e^{i a \ln |\xi|} - 1), \]
the asymptotic expression \eqref{est:l8_2} of $\tilde I(S)$, tells us that for $\xi \gg 1$,
\begin{align*}
\Psi(S)(\xi)&= c  +\frac{3i}{2\pi}\alpha -\frac{3i\e}{4\pi^2}\int_0^\xi \tilde I(S)(\eta)d\eta \\&= c+\frac{3i}{2\pi}\alpha -\frac{\e}{4\pi^2} \left(3i\mathcal{I}(S) - 3\frac{E}{a} + 3\frac{E}{a} e^{ia\ln|\xi|}\right) + O(|\xi|^{-1+\gamma/2}),
\end{align*}
with
\begin{align} \label{def:cal_I}
\mathcal{I} = \int_0^1 \tilde I(S)(\eta)d\eta + \int_1^\infty \left(\tilde I(S)(\eta) - E \frac{e^{ia \ln |\eta|}}{|\eta|}\right)d\eta.
\end{align}
Now we match the asymptotics of $S$ and $\Psi(S)$ at infinity, for the oscillating $e^{i a \ln |\xi|}$ term: 
\begin{equation}\label{eq:defia}
-\frac{3E\e}{4\pi^2 a}=A \iff a=-\frac{3\e}{4\pi} |A|^2.
\end{equation}

Moreover, we also match the two oscillating terms of $\partial_\xi S$ and $\partial_\xi \Psi(S)=-(3i\e/4\pi^2)I(S,S,S)$, for $\xi \gg 1$: this gives
\begin{equation}\label{eq:compat}
\beta=-8/9,\quad 4\pi^2 ia A = -3iE\e, \quad  3 i \beta B=-\frac{3i\e}{4\pi^2}F.
\end{equation}
The last condition defines $B$ and the second one is already guaranteed. In fact, the conditions on the derivative are truly the structural ones, while the remaining conditions on the function itself relate to constants of integration.

\begin{nb}
The above relation between $a$ and $A$ is also present in the work of Hayashi and Naumkin \cite{HN99}. Indeed, we can infer from their computations that, up to a specific phase correction (depending only on the modulus of the solution), the self-similar profile converges, as $t\to \infty$, to a fixed function. This implies that the phase correction, in our case, is given by $e^{ia\ln|\xi|}$. Since we assumed that the self-similar solution has, asymptotically, modulus equal to $|A|$, one may use the formula of Hayashi and Naumkin to deduce the relation $\ds a=\frac{3}{4\pi} |A|^2$.
\end{nb}

Summing up, our ansatz $S$ now only depend on $A$, and we will denote it $S_A$: it is given for $\xi \ge 0$ by
\begin{gather} \label{def:S_A}
S_A(\xi) := \chi(\xi) e^{i a\ln \xi} \left( A + B e^{2ia \ln |\xi|} \frac{e^{-i \frac{8}{9} \xi^3}}{\xi^3} \right), \quad S_A(-\xi) = \overline{S_A(\xi)},
\end{gather}
where
\begin{gather} \label{def:aB}
a=a(A) := - \frac{3\e}{4\pi} |A|^2, \quad B=B(A) := \frac{3}{16 \pi \sqrt 2} e^{ia \ln 3} |A|^2 A.
\end{gather}
With these definitions, observe that we can reformulate Lemma \ref{lem:desenvolveISSS} as
\begin{align} \label{est:I-S_A}
\forall \xi \in \m R, \quad \left| -\frac{3i\e}{4\pi^2}  \tilde I(S_A)(\xi) - S_A'(\xi) \right| \lesssim \min(1,|\xi|^{-2+\gamma/2}).
\end{align}

Matching the constants is more delicate, because the fixed point is of the form $S+z$: although the small remainder $z$ will not affect the oscillating terms, it does affect the constants $c$ and $\alpha$.

More precisely,  given $c,\alpha \in \m R$, our goal is to find $A \in \m C$ and a function $z$ such that $S_A+z$ is a fixed point of $\Psi = \Psi_{c,\alpha}$ (the map $\Psi$ is defined in \eqref{def:Psi} in terms of $c,\alpha$; it is convenient in this Section to make this dependence explicit). Matching the constants in the asymptotic for $S_A+z$ (which is 0) and for
$\Psi(S_A+z)$ yields
\begin{gather}
c+ \frac{3i}{2\pi} \alpha - \frac{\e}{4\pi^2} (3i \q I(A,z) - A) =0 \quad \text{where}  \nonumber \\
\q I(A,z) := \int_0^1 \tilde I(S_A+z)(\eta)d\eta + \int_1^\infty \left( \tilde I(S_A+z)(\eta) - \pi |A|^2 A \frac{e^{ia \ln |\eta|}}{|\eta|}\right)d\eta. \label{def:I(A,z)}
\end{gather}
Taking real and imaginary part in the above relation, we want to solve the system 
\begin{equation} \label{eq:sistemaC1}
c = -\e \Re A  -\frac{3\e }{4\pi^2} \Im \mathcal{I}(A,z) \quad \text{and} \quad
\alpha= -\frac{2\pi\e}{3} \Im A + \frac{\e }{2\pi} \Re \q I(A,z),
\end{equation}
(and $\Psi_{c,\alpha} (S_A+z) = S_A+z$).

\bigskip

In the remainder of this section, we will complete the proof of Theorem \ref{th1} by solving the fixed point equation, and the implicit system \eqref{eq:sistemaC1}. We proceed in the following way. 

First, we assume $A \in \m C$ is given, and we construct a fixed point for the function
\begin{gather} \label{def:fp_Psi}
z \mapsto \Psi_{c(A,z), \alpha(A,z)}(S_A+z) -S_A\end{gather}
where $\Psi_{c,\alpha}$ is defined in \eqref{def:Psi} and $c(A,z)$ and $\alpha(A,z)$ are \emph{defined} by \eqref{eq:sistemaC1}. We denote this fixed point $z_A$.

Second, we prove that the map $A \mapsto (c(A,z_A), \alpha(A,z_A))$ is bijective locally around 0 (heuristically, it is because $\q I(A,z)$ is cubic in $A$, $z$). Given $c$ and $\alpha$, its inverse provides the amplitude $A$ to define the ansatz, and thus desired self-similar profile.

\bigskip

We now define the functional spaces for $z$ and some multilinear estimates in the following Section \ref{sec:5.2}, before completing these two steps in Section \ref{sec:5.3}.

\subsection{Functional spaces for the fixed point} \label{sec:5.2}

Thus we are left with the fixed point equation, and the implicit system \eqref{eq:sistemaC1} relating $c, \alpha$ on one side and $A$ on the other side.

With the choice of ansatz \eqref{def:S_A}-\eqref{def:aB} above, we try to set up a fixed point argument. We take a remainder $z$ such that
\begin{equation}\label{eq:decayresto}
\begin{cases}
\ds z(\eta) = c + \frac{3i}{2\pi} \alpha+ O(|\eta|),\ z'(\eta)=O(1),& \text{for } 0<\eta<1,\\
z(\eta)=O(|\eta|^{-k}), \ z'(\eta)=O(|\eta|^{-k-1}) & \text{for } \eta>1.
\end{cases}
\end{equation}
We want to choose $k$ in such a way that the remainder in the matching between $S_A$ and $\Psi(S_A)$ satisfies the above properties. It will turn out that $k=1-\gamma/2\in (1/2,4/7)$ works. 
 
The analysis will be carried out in the space $Z^k$; we will also use a slightly different quantity, which handles low frequencies more precisely:
\[
|z|_{k,c+\frac{3i}{2\pi}\alpha} :=\|(z-c-\frac{3i}{2\pi}\alpha)|\eta|^{-1}\|_{L^{\infty}(0,1)} +  \|z|\eta|^k\|_{L^{\infty}(1,\infty)} +\|(1+|\eta|^{k+1})z'\|_{L^{\infty}(\real^+)}.
\]
We then look for a fixed point of \eqref{def:fp_Psi} 
over the set $\{z\in Z^k: |z|_{k,c+\frac{3i}{2\pi}\alpha}<\epsilon \}$, for some small $\epsilon>0$.

First of all, writing
\begin{gather*} I(f,g,h)(\xi) =  \frac{1}{2}\int_\eta  e^{-3i \Phi(\xi,\eta)} \bar h(\eta- \xi) \left( \int_\nu e^{ \frac {3i} 4 \eta \nu^2} f \left( \frac{\eta + \nu}{2} \right) g \left( \frac{\eta - \nu}{2} \right) d\nu \right) d\eta \\
\text{where} \quad \Phi(\xi,\eta) =  \eta \xi^2 - \xi \eta^2 + \frac 1 4  \eta^3,
\end{gather*}
one has
\[ I(f,g,\bar{h})=I(g,f,\bar{h})=I(h,g,\bar{f}), \] 
which is easily seen in the variables $\eta_1,\eta_2$ and $\eta_3$. Hence all we need to estimate are the:
\begin{enumerate}
\item Linear term: $I(z,S_A,S_A)$;
\item Quadratic term: $I(S_A,z,w)$;
\item Cubic term: $I(z,w,u)$;
\end{enumerate}
We choose these arrangements so that no term of the form $K(S_A,z)$ appears and put different remainders keeping in mind that we will need to prove that $\Psi$ is a contraction.

\begin{lem}\label{lem:desenvolveIzw}
Let $z,w\in Z^k$. Then
\begin{equation} \label{est:l9_1}
|K(z,w)(\eta)|\lesssim \| z \|_{Z^k}\| w \|_{Z^k},\quad |\partial_\eta K(z,w)|\lesssim \frac{\| z \|_{Z^k}\| w \|_{Z^k}}{|\eta|},\quad \mbox{for }|\eta|<1,
\end{equation}
and
\begin{equation}  \label{est:l9_2}
|K(z,w)(\eta)|\lesssim \frac{\| z \|_{Z^k}\| w \|_{Z^k}}{|\eta|^{k+1}},\quad |\partial_\eta K(z,w)|\lesssim \frac{\| z \|_{Z^k}\| w \|_{Z^k}}{|\eta|^{k}},\quad \mbox{for }|\eta|>1,
\end{equation}
\end{lem}

\begin{proof}

\emph{Proof of \eqref{est:l9_2} for $K(z,w)$.} We start with $|\eta|>10$. Then
\begin{align*}
\sqrt{\eta}K(z,w)(\eta)&=\int_{|\mu|\le |\eta|^{3/2}/2} e^{3i\mu^2/4}z\left(\eta+\frac{\mu}{\sqrt{\eta}}\right)w\left(\eta-\frac{\mu}{\sqrt{\eta}}\right)d\mu\\&\quad + \int_{|\mu|\le |\eta|^{3/2}/2} e^{3i\mu^2/4}z\left(\eta+\frac{\mu}{\sqrt{\eta}}\right)w\left(\eta-\frac{\mu}{\sqrt{\eta}}\right)d\mu.
\end{align*}
In the region $|\mu|\le |\eta|^{3/2}/2$, we write
\[
\int_{|\mu|\le |\eta|^{3/2}/2} e^{3i\mu^2/4}z\left(\eta+\frac{\mu}{\sqrt{\eta}}\right)w\left(\eta-\frac{\mu}{\sqrt{\eta}}\right)=|\eta|^{3/2}\int_{|\nu|<1/2} e^{3i\eta^3\nu^2/4}z(\eta(1+\nu))w(\eta(1-\nu))d\nu
\]
and split the integral at $\nu=\ell$:
\begin{align*}
\left|\int_{|\nu|< \ell} \right| & \lesssim \int_{|\nu|<\ell} |\eta|^{-2k}d\nu = a|\eta|^{-2k}, \\
\int_{\ell<\nu<1/2}  &\lesssim \int_{|\nu|<1/2} \frac{\eta^3\nu}{\eta^3\nu}e^{3i\eta^3\nu^2/4}z(\eta(1+\nu))w(\eta(1-\nu))d\nu \\
& = O(|\eta|^{-3-2k} \ell^{-1}) - \int_\ell^{1/2} e^{3i\eta^3\nu^2/4}\Bigg(-\frac{1}{\eta^3\nu^2}z(\eta(1+\nu))w(\eta(1-\nu)) \\ & \qquad + \frac{1}{\eta^2\nu}\left(z'(\eta(1+\nu))w(\eta(1-\nu)) + z(\eta(1+\nu))w'(\eta(1-\nu))\right)\Bigg)d\nu\\
& = O(|\eta|^{-3-2k} \ell^{-1}) + O(|\eta|^{-3-2k}\ln \ell).
\end{align*}
We now choose $\ell=|\eta|^{-3/2}$, which implies that the contribution of the region $\{|\mu|\le |\eta|^{3/2} \}$ is $O(|\eta|^{-2k})$.

For the region $|\mu|\ge |\eta|^{3/2}$,
\begin{align*}
\MoveEqLeft \int_{|\mu|\ge |\eta|^{3/2}/2} e^{3i\mu^2/4}\mu\frac{1}{\mu}z\left(\eta+\frac{\mu}{\sqrt{\eta}}\right)w\left(\eta-\frac{\mu}{\sqrt{\eta}}\right) d\mu \\ 
& = O(|\eta|^{-3/2}) + \int_{|\mu|\ge |\eta|^{3/2}/2} e^{3i\mu^2/4}\Bigg(-\frac{1}{\mu^2}z\left(\eta+\frac{\mu}{\sqrt{\eta}}\right)w\left(\eta-\frac{\mu}{\sqrt{\eta}}\right) \\ 
& \qquad + \frac{1}{\mu\sqrt{\eta}} z'\left(\eta+\frac{\mu}{\sqrt{\eta}}\right)w\left(\eta-\frac{\mu}{\sqrt{\eta}}\right) - \frac{1}{\mu\sqrt{\eta}}z\left(\eta+\frac{\mu}{\sqrt{\eta}}\right)w'\left(\eta-\frac{\mu}{\sqrt{\eta}}\right)\Bigg)d\mu
\end{align*}
Without loss of generality, we look at the region where $\mu$ has the same sign as $\eta$, so that the contribution is bounded by
\begin{align*}
\MoveEqLeft O(|\eta|^{-3/2}) + \int_{|\mu|>|\eta|^{3/2}/2} \frac{1}{\mu\sqrt{\eta}}|\mu/\sqrt{\eta}|^{-k-1} + \frac{1}{\mu\sqrt{\eta}}|\mu/\sqrt{\eta}|^{-k} d\mu + \left(\frac{1}{\mu}\left|z\left(\eta+\frac{\mu}{\sqrt{\eta}}\right)\right|\right)\Bigg|_{\mu=|\eta|^{3/2}}\\
 & =  O(|\eta|^{-3/2}) + O(|\eta|^{-3/2-k}) + O(|\eta|^{-1/2-k})
\end{align*}
Hence
\[
|K(z,w)|=O(|\eta|^{-1-k}), \quad |\eta|>1.
\]

\bigskip

\emph{Proof of \eqref{est:l9_1} for $K(z,w)$.} Now we consider the case $|\eta|<10$. We split the integral $K$ at $|\nu|=20$. For $|\nu|<20$, everything is bounded and so 
\[
\int_{|\nu|<20} e^{3i\eta\nu^2/4}z\left(\frac{\eta+\nu}{2}\right) w\left(\frac{\eta-\nu}{2}\right)d\nu=O(1).
\]
In the region $|\nu|>20$, we use the decay of $z$ and $w$ to obtain
\[
\left|\int_{|\nu|>20}\right| \lesssim \int_{|\nu|>20} |\nu|^{-2k} d\nu = O(1).
\]
(Here we used $k>1/2$ but this part could be dealt with for $k$ smaller). This completes the estimates for $K(z,w)$.

\bigskip

We now turn to the derivative estimates $\partial_\eta K(z,w)$. We compute for $\eta>0$:
\begin{align*}
\partial_\eta K(z,w) &= -\frac{1}{2\eta}K(z,w) \\
& \quad +\frac{1}{2\eta}\int e^{3i\eta\nu^2/4}\left(\eta-\frac{\nu}{2}\right)\left(z'\left(\frac{\eta+\nu}{2}\right)w\left(\frac{\eta-\nu}{2}\right) + z\left(\frac{\eta+\nu}{2}\right)w'\left(\frac{\eta-\nu}{2}\right)\right)d\nu
\end{align*}
It is enough to treat the integral term with derivative in $z$.

\bigskip

\emph{Proof of \eqref{est:l9_2} for $\partial_\eta K(z,w)$.} We start with $\eta>10$. The jump term occurs at $\nu = - \eta$ and is
\[ e^{3i \eta^3/4} \frac{\eta}{2} \frac{3i}{2\pi} \alpha w(\eta) = O(|\eta|^{1-k}). \]
In the region $|\nu|<|\eta|/2$, we simply estimate
\[
\left|\int_{|\nu|\le |\eta|/2}\right|\lesssim \int_{|\nu|\le |\eta|/2} |2\eta-\nu||\eta|^{-2k-1} d\nu = O(|\eta|^{1-2k}).
\]
In the region $|\nu|>|\eta|/2$, if $\nu$ has the opposite sign as $\eta$ (hence negative),
\begin{align*}
\int_{|\nu|> |\eta|/2} & = \int_{-\infty}^{-2\eta} + \int_{-2\eta}^{-\eta-1} + \int_{-\eta-1}^{-\eta} + \left(e^{3i\eta\nu^2/4}\left(\eta-\frac{\nu}{2}\right)w\left(\frac{\eta-\nu}{2}\right)\right)\Bigg|_{\nu=\eta}  + \int_{-\eta}^{-\eta+1} + \int_{-\eta+1}^{-\eta/2} \\ & \lesssim \int_{-\infty}^{-2\eta} |\eta|^{1-k}|\eta+\nu|^{-k-1}d\nu + \int_{-2\eta}^{-\eta-1}|\eta|^{1-k}|\eta+\nu|^{-k-1}d\nu + \int_{-\eta-1}^{-\eta}|\eta|^{1-k} d\nu \\&\qquad + O(|\eta|^{1-k})+ \int_{-\eta}^{-\eta+1}|\eta|^{1-k} d\nu + \int_{-\eta+1}^{-\eta/2}|\eta|^{1-k}|\eta+\nu|^{-k-1}d\nu \\ & \lesssim \int_{-\infty}^{-2\eta}|\nu|^{-2k}d\nu + |\eta|^{1-k}\int_{-\eta}^{-1}|\nu|^{-1-k}d\nu + O(|\eta|^{1-k}) + |\eta|^{1-k}\int_{1}^{\eta/2} |\nu|^{-1-k}d\nu \\
& \lesssim O(|\eta|^{1-k}).
\end{align*}
Here we really need $k>1/2$ to ensure convergence. When $\nu$ has the same sign as $\eta$, the decays are stronger. Thus, with the prefactor $\frac{1}{2\eta}$
\[
|\partial_\eta K(z,w)| = O(|\eta|^{-k}), \quad |\eta|>1.
\]

\bigskip

\emph{Proof of \eqref{est:l9_1} for $\partial_\eta K(z,w)$.} For $|\eta|<10$, we split the integral at $|\nu|=20$. For $|\nu|<20$, we bound directly and obtain $O(1)$. For $|\nu|>20$,
\[
\left|\int_{|\nu|>20}\right| \lesssim \int_{|\nu|>20} |\nu||\nu|^{-2k-1} d\nu = O(1).
\]
(We used again $k>1/2$, even though it might be dealt with in some other way). Hence for $|\eta| \le 1$,
\[
|\partial_\eta K(z,w)| = O(|\eta|^{-1}). \qedhere
\]
\end{proof}

\begin{lem}\label{lem:linquadcub}
Let $z,w,u\in Z^k$ and $A \in \m C$, $|A| < 1$. Then
\begin{enumerate}
\item (Linear estimate)
\[
| I(z,S_A,S_A)(\xi)|\lesssim |A|^2\| z \|_{Z^k} \min\{1, |\xi|^{-k-1}\}
\]
\item (Quadratic estimate)
\[
| I(S_A,z,w)(\xi)|\lesssim |A| \| z \|_{Z^k}\| w \|_{Z^k} \min\{1, |\xi|^{-k-1}\}
\]
\item (Cubic estimate)
\[
| I(z,w,u)(\xi)|\lesssim \| z \|_{Z^k}|\| w \|_{Z^k} \| u \|_{Z^k} \min\{1, |\xi|^{-k-1}\}
\]
\end{enumerate}
\end{lem}

\begin{proof}
First, notice that, because of the definition of $B$ in terms of $A$, all bounds on $S_A$ are linear in $A$ and all bounds on $K(S_A,S_A)$ are quadratic in $A$. Since
\begin{gather*}
I(z,S_A,S_A)=J(K(S_A,S_A),z),\quad I(S_A,z,w)=J(K(z,w),S_A), \\
I(z,w,u)=J(K(w,u),z),
\end{gather*}
the claimed estimates follow from direct application of Lemmas \ref{lem:primeiro} to \ref{lem:regsing}, using the estimates of Lemmas \ref{lem:desenvolveISS} and \ref{lem:desenvolveIzw}.
\end{proof}

\subsection{Proofs of the main results} \label{sec:5.3}

We define the maps $c,\alpha: \m C \times Z^k \to \m R$ (as explained in Section \ref{sec:5.1}) by
\begin{gather} \label{eq:sistemathetac} 
c(A,z) = -\e\Re A  -\frac{3\e }{4\pi^2} \Im \mathcal{I}(A,z) \quad \text{and} \quad
\alpha(A,z) = -\frac{2\pi\e}{3} \Im A + \frac{\e}{2\pi} \Re \q I(A,z),
\end{gather}
where $\q I(A,z)$ is defined in \eqref{def:I(A,z)}.

\begin{lem}\label{lem:escolhac}
For all $(A,z) \in \m C \times Z^k$, $|A| < 1$, we have 
\[ |c(A,z)-c(A,w)|+|\alpha(A,z)-\alpha(A,w)|\lesssim (|A|^2 + \| z \|_{Z^k}^2 + \| w \|_{Z^k}^2) \| z-w \|_{Z^k}. \]
\end{lem}

\begin{proof}
Observe that the term $\mathcal{I}(A,z)$ is cubic in $z$ and $A$, as expressed in Lemma \ref{lem:linquadcub}. Hence
\begin{align*}
\MoveEqLeft |c(A,z)-c(A,w)|+|\alpha(A,z)-\alpha(A,w)| \lesssim |\mathcal{I}(A,z)-\mathcal{I}(A,w)|\\
& \lesssim \left|\int_0^\infty (\tilde{I}(S_A+z)(\xi)-\tilde{I}(S_A+w)(\xi)) d\xi\right| \\  
& \lesssim (|A|^2+\| z \|_{Z^k}^2+\| w \|_{Z^k}^2)\| z-w \|_{Z^k}. \qedhere
\end{align*}
\end{proof}

The next result constructs the fixed point of the map
\begin{gather} \label{def:Phi_A}
\tilde \Psi_A: z \mapsto \Psi_{c(A,z), \alpha(A,z)}(S_A + z)-S_A
\end{gather}
for any given $A \in \m C$ small. (In the next results, do not confuse the small parameters $\epsilon$ or $\epsilon_1$ with the signum $\varepsilon$).

\begin{thm} \label{th2}
Fix $k\in (1/2, 4/7)$. For $A \in \complex$, $|A| < \epsilon_1$ sufficiently small, the map $\tilde \Psi_A$ admits a (unique) fixed point which we denote $z_A \in Z^k$, and such that
\[ |z_A|_{k,c(A,z_A)+\frac{3i}{2\pi}\alpha(A,z_A)} < 3|A|. \]
In other words the function $v:=S_A+z_A$ satisfies for $\xi >0$
\[
v(\xi)=c(A,z_A)+\frac{3i}{2\pi}\alpha(A,z_A) -\frac{3i\e}{4\pi^2 }\int_0^\xi \tilde{I}(v)(\eta)d\eta
\]
and $v(-\xi) = \overline{v(\xi)}$.
\end{thm}

\begin{proof}
In this proof only, the implicit constants in the $O$ are absolute. Fix $M>0$ and define
\[
E=\{z\in Z^k: \| z \|_{Z^k}\le M \}
\]
endowed with the distance $d(z,w)=\| z-w \|_{Z^k}$. It is trivial to check that $(E,d)$ is a complete metric space. From the definition \eqref{def:Phi_A} and \eqref{def:Psi}, for $z\in Z^k$ and $\xi >0$
\[
\tilde{\Psi}_A (z)(\xi)=c(A,z)+\frac{3i}{2\pi}\alpha(A,z)-\frac{3i\e}{4\pi^2 }\int_0^\xi \tilde{I}(S_A + z)(\eta)d\eta - S_A(\xi).
\]
Then the matching asymptotics of $\Psi(S)$ and $S$ and the estimates of Lemma \ref{lem:linquadcub} imply that, for $0<\xi<1$,
\begin{align*}
\tilde \Psi_A(z)(\xi) & = c(A,z) + \frac{3i}{2\pi}\alpha(A,z) + (\| z \|_{Z^k}^3+|A|^3 + |A|)O(|\xi|)\\
\tilde \Psi(z)'(\xi) &= -\frac{3i\e}{4\pi^2 } \tilde I(S_A+z)(\xi) - S_A'(\xi)= (\| z \|_{Z^k}^3+|A|^3 + |A|)O(1).
\end{align*}
For $\xi>1$,
\begin{align*}
\tilde \Psi_A(z)(\xi) &= c(A,z)+\frac{3i}{2\pi}\alpha(A,z)-\frac{3i\e}{4\pi^2} \int_0^\xi \tilde I(S_A+z)(\eta)d\eta - S(\xi) \\
&= c(A,z)+\frac{3i}{2\pi}\alpha(A,z)-\frac{3i\e}{4\pi^2}\mathcal{I}(A,z) +\frac{3i\e}{4\pi^2}\int_\xi^\infty \left( \tilde{I}(S_A+z)(\eta ) - E \frac{e^{ia \ln |\eta|}}{\eta} \right)d\eta \\&\qquad-\frac{3i\e}{4\pi^2} \int_1^\xi E\frac{e^{ia\ln|\eta|}}{|\eta|}d\eta - S_A(\xi).
\end{align*}
Now, from Lemma \ref{lem:linquadcub},
\[ |\tilde I(S_A+z)(\eta) - \tilde I(S_A)(\eta)| = (|A|^2 \| z \|_{Z^k} + \| z \|_{Z^k}^3) O(|\xi|^{-k-1}), \]
so that
\[ \int_{\xi}^{+\infty} (\tilde I(S_A+z)(\eta) - \tilde I(S_A)(\eta)) d\eta = (|A|^2 \| z \|_{Z^k} + \| z \|_{Z^k}^3) O(|\xi|^{-k}). \] 
When integrating \eqref{est:l8_2}, the second, highly oscillating term is negligeable so that
\[ \int_{\xi}^{+\infty} \left( \tilde I(S_A)(\eta) - E \frac{e^{ia \ln |\eta|}}{\eta} \right) d\eta = |A|^3 O(|\xi|^{-1+\gamma/2}) =  |A|^3 O(|\xi|^{-k}). \]
Also,
\begin{align*}
-\frac{3i\e}{4\pi^2} \int_1^\xi E\frac{e^{ia\ln|\eta|}}{|\eta|}d\eta & = - A + A e^{ia \ln |\xi|}, \\
S_A(\xi) & = A e^{ia \ln |\xi|} + A O(|\xi|^{-1}).
\end{align*}
Combining the above and using the cancellation due to the definition \eqref{eq:sistemathetac}, we get
\begin{gather}
\tilde \Psi_A(z)(\xi) = (\| z \|_{Z^k}^3+|A|^3+|A|)O(|\xi|^{-k} ).
\end{gather}
Similarly,
\begin{equation}
\tilde \Psi_A(z)'(\xi) = \frac{3i}{4\pi^2} \tilde I(S_A+z)(\xi) - S'(\xi)= (\| z \|_{Z^k}^3+|A|^3 + |A|)O(|\xi|^{-k-1}).
\end{equation}

We now turn to difference estimates. As $A$ is \emph{fixed}, using 
the estimates of Lemmas \ref{lem:escolhac} and \ref{lem:linquadcub}, one easily shows that
\begin{align*}
|\tilde{\Psi}_A(z)(\xi)-\tilde{\Psi}_A(w)(\xi)| & \le (\| z \|_{Z^k}^2 + \| w \|_{Z^k}^2 + |A|^2)\| z-w \|_{Z^k}O(\min\{1, |\xi|^{-k} \}) \\
|\tilde{\Psi}_A(z)'(\xi)-\tilde{\Psi}_A(w)'(\xi)| & \le (\| z \|_{Z^k}^2 + \| w \|_{Z^k}^2 + |A|^2)\| z-w \|_{Z^k}O(\min\{1, |\xi|^{-k-1} \})
\end{align*}
Hence, for any $z,w\in E$,
\[
|\tilde{\Psi}_A(z)|_k \lesssim |A| + |A|^3 + \| z \|_{Z^k}^3\lesssim |A| + M^3
\]
and
\[
d(\tilde{\Psi}_A(z), \tilde{\Psi}_A(w) ) \lesssim (\| z \|_{Z^k}^2 + \| w \|_{Z^k}^2 + |A|^2)d(z,w) \lesssim (|A|^2 + M^2)d(z,w).
\]
Therefore, for $M=2|A|$ and $|A| < \epsilon_1$ sufficiently small, $\tilde{\Psi}_A:E\mapsto E$ is a strict contraction. By Banach's fixed point theorem, there exists a unique $z_A \in E$ such that
\[
v:=S_A+z_A=S_A+\tilde{\Psi}_A(z)=S_A+\Psi_A(S+z)-S_A=\Psi_A(v).
\]
It remains to see that $|z|_{k,c(A,z)+\frac{3i}{2\pi}\alpha(A,z)}< 3|A|$. We already know that $\| z \|_{Z^k} \le 2|A|$; the rest follows from the fact that, for $0<\xi<1$,
\begin{align*}
z_A(\xi) & =c(A,z_A)+\frac{3i}{2\pi}\alpha(A,z_A)-\frac{3i\e}{4\pi^2}\int_0^\xi \tilde{I}(S+z)(\eta)d\eta \\
& = c(A,z_A) +\frac{3i}{2\pi}\alpha(A,z_A) + (\| z \|_{Z^k}^3 + |A|^3) O(|\xi|). \qedhere
\end{align*}
\end{proof}

We now complete the proof of Theorem \ref{th1}, by reverting the roles of $(c,\alpha)$ and $A$. Fix $k \in (\frac{1}{2}, \frac{4}{7})$ until the end of this section. We first prove some Lipschitz continuity of the maps $A \mapsto \tilde I(A,z)$ and $A \mapsto z_A$.

Introduce for convenience of notation the remainder term in $\tilde I(S_A+z)$:
\begin{gather}
R(A,z)(\xi): = -\frac{3i\e}{4\pi^2} \tilde I(S_A+z) - S_A'(\xi).
\end{gather}
The estimate \eqref{est:I-S_A} gives decay on $R(A,0)$, and in the next lemma we claim a difference estimate.

\begin{lem} \label{lem:21}
Let $\epsilon>0$ small enough, and $A_1,A_2 \in \m C$ such that $|A_1|,|A_2| \le \epsilon$, and $z\in Z^k$ such that $\| z \|_{Z^k} \le 3 \epsilon$. Then for all $\xi \in \m R$,
\[ |R(A_1,z)(\xi) - R(A_2,z)(\xi)| \lesssim \epsilon^2 |A_1-A_2| \ln^3 (2+|\xi|) O(\min(1,|\xi|^{-1-k}) . \]
\end{lem}

\begin{proof}[Sketch of the proof]
$R(A,z)(\xi)$ is given by a sum of integrals which, after the appropriate integration by parts, can all be estimated directly with absolute values on the integrand. Regarding the dependence on $A$ for these integrals, when it appears in the amplitude constants, we can directly estimate the difference and obtain a $|A_1-A_2|$ factor together with the same decay (by the same computations, done in Sections \ref{sec:3} and \ref{sec:4}). The ``worst'' dependence on $A$ is when it occurs in the phases; observe that this only happens through $a(A) = (3/4\pi) |A|^2$ in the oscillating term $e^{i a \ln |\xi|}$ (the key is that in the highly oscillating terms with phase $e^{-i 8/9 \xi^3}$, there is no dependence on $A$: $\beta$ is independent of $A$!). This leads to terms of the form
\begin{gather} \label{est:exp_phase_diff}
\left| e^{ia(A_1)\ln |\eta|}-e^{ia(A_2)\ln |\eta|} \right|\lesssim |a(A_1)-a(A_2)|\ln |\eta| \lesssim |A_1- A_2|\ln (2+|\eta|).
\end{gather}
As a consequence we obtain the claimed estimate.
\end{proof}

So there is a logarithmic loss when performing difference estimates. However, this can be compensated by decreasing slightly the parameter $k$, which controls the decay rate in $Z^k$, and so we recover Lipschitz continuity for the maps we are interested in.

\begin{lem} \label{lem:22}
For any $\epsilon ,\delta>0$ sufficiently small, the following holds true. Let $A_1, A_2 \in \m C$ with $|A_1|, |A_2| <\epsilon$, and $z,w \in Z^k$ such that $\| z \|_{Z^k}, \| w \|_{Z^k} \le 3\epsilon$ then
\begin{align}
|\q I(A_1,z) - \q I(A_2,w) | & \le C \epsilon^2 (|A_1 - A_2| + \| z-w \|_{Z^k}), \label{est:I(A,z)_lip} \\
\| z_{A_1} - z_{A_2} \|_{Z^{k-\delta}} & \le C(\delta) \epsilon^2 |A_1 - A_2|, \label{est:z_A_lip}
\end{align}
where $\mathcal{I}$ is as in \eqref{def:I(A,z)} and $C(\delta)$ only depends on $\delta$.
\end{lem}

\begin{proof}

\emph{Proof of estimate \eqref{est:I(A,z)_lip}}. Using lemma \ref{lem:linquadcub}, we have
\begin{align*}
|\q I(A_2,z) - \q I(A_2,w)| & \le \int_0^{+\infty} |\tilde I(S_{A_2}+z)(\eta) - \tilde I(S_{A_2}+w)(\eta)| d\eta \\
& \lesssim (|A_2|^2 + \| z \|_{Z^k}^2 + \| w \|_{Z^k}^2) \| z-w \|_{Z^k},
\end{align*}
which gives Lipschitz continuity of $\q I$ with respect to the $z$ variable. For the $A$ variable, we have by definition, 
\[ \q I(A,z) = \int_0^1 S_A'(\xi) d\xi + \int_0^{+\infty} R(A,z)(\xi)) d\xi + \int_1^{+\infty} F \frac{e^{3ia \ln |\xi|} e^{-8i\xi^3/9}}{\xi} d\xi \]
and the three integrals are convergent. Hence
\begin{align*}
|\q I(A_1,z) - \q I(A_2,z) | & \le  \int_0^1 |S_{A_1}'(\xi) - S_{A_2}'(\xi)| d\xi + \int_0^{+\infty} |R(A_1,z)(\xi) - R(A_2,z)(\xi)| d\xi \\
& \qquad + \left|  \int_1^{+\infty} e^{-8i\xi^3/9} \left( F(A_1) \frac{e^{3ia(A_1) \ln |\xi|}}{\xi} -  F(A_2) \frac{e^{3ia(A_1) \ln |\xi|}}{\xi}\right)  d\xi \right|
\end{align*}
The first two terms are $O(\epsilon^2 |A_2-A_1|)$ (using Lemma \ref{lem:21} for the second). The last term is explicit, by performing an integration by parts, we see that, as in \eqref{est:exp_phase_diff}, it is bounded by
\begin{align*}
\int_1^{+\infty} \left| F(A_1) \frac{e^{3ia(A_1) \ln |\xi|}}{\xi^4}  -  F(A_2) \frac{e^{3ia(A_1) \ln |\xi|}}{\xi^4} \right| d\xi & \lesssim \epsilon ^2 |A_1-A_2| \int_1^{+\infty}  \frac{\ln (2+|\xi|)}{\xi^4} d\xi \\
& \lesssim \epsilon ^2 |A_1-A_2|.
\end{align*}

\emph{Proof of estimate \eqref{est:z_A_lip}}.
Recall that 
\[ z_A' = -\frac{3i\e}{4\pi^2}  \tilde I(S_A + z_A). \]
By Lemmas \ref{lem:linquadcub} and \ref{lem:21}, we infer that for $k-\delta>1/2$, 
\begin{align}
| z_{A_1}'(\xi) - z_{A_2}'(\xi)| & \le \frac{3}{4\pi^2} \left( | \tilde I(S_{A_1} + z_{A_1})(\xi) - \tilde I(S_{A_1} + z_{A_2})(\xi)| \right. \nonumber \\
& \qquad \left. + | \tilde I(S_{A_1} + z_{A_2})(\xi) - \tilde I(S_{A_2} + z_{A_2})(\xi)|  \right) \nonumber \\
 & \lesssim \epsilon^2 |z_{A_1} - z_{A_2}|_{k-\delta} O(\min(1,|\xi|^{-1-k+\delta})) \nonumber \\
 & \qquad +  \epsilon^2 |A_1-A_2| \ln^3 (2+|\xi|) O(\min(1,|\xi|^{-1-k})) \label{est:z'}
\end{align}
Hence
\[ \| (1+|\eta|^{k-\delta+1} )(z_{A_1} - z_{A_2})(\eta) \|_{L^\infty} \le \epsilon^2 |z_{A_1} - z_{A_2}|_{k-\delta} + C(\delta) \epsilon^2 |A_1-A_2|. \]
(actually one can choose $C(\delta) \sim 1/\delta$). This is the derivative estimate. For the function estimate, it suffices to integrate \eqref{est:z'}, using the fact that $z_{A_1}(\xi), z_{A_2} (\xi) \to 0$ as $\xi \to +\infty$. This gives, for $\xi >0$
\begin{align*}
 |z_{A_1}(\xi) - z_{A_2}(\xi)| & \lesssim \epsilon^2 |z_{A_1} - z_{A_2}|_{k-\delta} O(\min(1,|\xi|^{-k+\delta})) \\
 & \qquad  + \epsilon^2 |A_1-A_2| \ln^3 (2+|\xi|) O(\min(1,|\xi|^{-k})). 
\end{align*}
By symmetry, this inequality also holds for $\xi <0$, and hence
\[ \| (1+ |\eta|^{k-\delta}) (z_{A_1} - z_{A_2})(\eta) \|_{L^\infty} \lesssim \epsilon^2 \| z_{A_1} - z_{A_2} \|_{Z^{k-\delta}} + C(\delta) |A_1-A_2|. \]
Summing up, we get for some constant $C$ independent of $A_1,A_2,\delta$,
\[ \| z_{A_1} - z_{A_2} \|_{Z^{k-\delta}} \le C \epsilon^2 \|z_{A_1} - z_{A_2} \|_{Z^{k-\delta}} + C(\delta) |A_1-A_2|. \]
Choosing $\epsilon$ so small that $C\epsilon^2 \le 1/2$, we get
\[ \| z_{A_1} - z_{A_2} \|_{Z^{k-\delta}} \le C(\delta) |A_1-A_2|. \qedhere \]
\end{proof}

We can now complete the proof the Theorem \ref{th1}.

\begin{proof}[Proof of Theorem \ref{th1}]
We consider the map
\[ f: \begin{array}[t]{r@{\ }c@{\ }l} 
\m C & \to & \m R^2 \\
A & \mapsto & \ds \begin{aligned}[t]
f(A) & = (c(A,z_A),\alpha(A,z_A)) \\
 & = \left(-\e\Re A  -\frac{3\e }{4\pi^2} \Re \mathcal{I}(A,z_A),  -\frac{2\pi\e}{3} \Im A + \frac{\e }{2\pi} \Re \q I(A,z_A) \right)
 \end{aligned}
\end{array} . \]
We claim that there exists $\epsilon_0 >0$ and a neighborhood $\q V$ of $0 \in 
\m C^2$ such that $f: \q V \to B$ is bijective and bi-Lipschitz, where $B$ is the open ball centered at $(0,0)$ of radius $\epsilon_0$ of $\m R^2$.

Observe that this means that given $(c,\alpha)$ such that $c^2+\alpha^2 < \epsilon_0$, there exist a unique $A \in \q V$ such that the compatibility condition \eqref{eq:sistemaC1} are fulfilled with $z=z_A$, and so $z_A$ is the sought for remainder.

If $f$ was $\q C^1$, we would merely apply the inverse function theorem, but our estimates do not quite reach this regularity. Actually, $f$ is a Lipschitz perturbation of the invertible $\m R$-linear map $L \in \q L(\m C,\m R^2)$ associated to the matrix $\begin{pmatrix}
1 & 0 \\
0 & \frac{2\pi}{3}
\end{pmatrix}$ (we identified $\m C$ and $\m R^2$). More precisely, fix $k \in ( \frac{1}{2}, \frac{4}{7})$ and $\delta >0$ so small that $k - \delta > \frac{1}{2}$, then Lemma \ref{lem:22} shows that the map $A \mapsto \q I(A,z_A)$ has Lipschitz constant $C(\delta) \epsilon_0^2$ on $B$. Hence the same is true for $g := f - L$, that is, for all $A_1,A_2 \in B$,
\[ \| g(A_1) - g(A_2) \| \le C(\delta) \epsilon_0^2 |A_1 - A_2|. \]

We use the following weakened version of the inverse function theorem.
\begin{claim} \label{cl:lit}
Let $(E,\| \cdot \|_E), (F,\| \cdot \|_F)$ be two Banach spaces and $L \in \q L(E,F)$ a (bi-) continuous invertible linear map. Consider $f = L +g$ where $g$ is a $c$-Lipschitz map defined on a neighbourhood of $0 \in E$, with values in $F$ and such that $g(0)=0$.

If $c< \| L^{-1} \|_{F \to E}^{-1}$, then there exists two open sets (containing $0$) $V$ of $E$, and $W$ of $F$ ($W$ can be chosen to be a ball centered at $0$), such that $f$ is bijective $V \to W$ and bi-Lipschitz, and $f^{-1}$ has Lipschitz constant less than
\[ \frac{1}{\| L^{-1} \|_{F \to E}^{-1}-c}. \]
\end{claim}
We apply this claim to $f$ and this concludes the proof of Theorem \ref{th1}.
\end{proof}

\begin{proof}[Proof of Proposition \ref{prop7}]
Either using a refined version of \cite[Lemma 2.1]{HN01} or applying directly some stationary phase arguments, one may show that the solution built in Theorem \ref{th1} satisfies in physical space
\[ 
V(y)\to 0 \quad \text{as} \quad y\to +\infty.
\]
Therefore, using the existence and uniqueness of decaying self-similar solutions (see \cite[Theorem 1]{HM80}), we conclude that our solution coincides with the solution $V_\kappa$ built in \cite{HM80}, for some $\kappa$. To see the precise relation between $A$ and $\kappa$, let us compute briefly the leading order term of $V$ when $y\to -\infty$. Since the second term in the anstatz and the remainder $z$ are in $L^2$, we have
\[
V(y) =\frac{1}{\pi}\text{Re}\int_0^\infty Ae^{iy\xi +\xi^3}e^{ia\ln \xi}\chi(\xi)d\xi\ +\ L^2 \text{-remainder}.
\]
A standard stationary phase argument shows that the main asymptotics are given by the contribution of the point $\xi_0=\sqrt{|y|/3}$. At this point, the phase $R(\xi)=y\xi + \xi^3$ is stationary and
\[
R(\xi_0)= - 2 |y/3|^{3/2},\quad R''(\xi_0)=2\sqrt{3|y|}.
\]
We then obtain, for $y \to -\infty$,
\begin{align*}
V(y)&=\frac{1}{\pi}\Re Ae^{i R(\xi_0)}\sqrt{\frac{2\pi}{R''(\xi_0)}} e^{ia\ln \xi_0}\ +\ L^2\mbox{-remainder}\\
&= \frac{|A|}{\sqrt{\pi}|3y|^{1/4}}\cos\left(- \frac{2}{3\sqrt 3} |y|^{3/2} + \frac{a}{2}\ln|y|+\theta_0\right) \ +\ L^2\text{-remainder},
\end{align*}
for some $\theta \in \m R$, and so
\[
|A|^2=2\ln\left(\frac{1}{1-\kappa^2}\right),\quad a=\frac{3|A|^2}{4\pi} = \frac{3}{2\pi}\ln\left(\frac{1}{1-\kappa^2}\right). 
\] 
Finally, it also follows from \cite{HM80} that  $\kappa$ is positive if and only if $V_\kappa$ has a positive average (meaning that $c>0$). Since $c$ and $\Re A$ have the same sign, the claimed result follows.
\end{proof}

\nocite{HM80}
\nocite{HN01}
\nocite{HN99}
\nocite{BV08}
\nocite{BV09}
\nocite{CV08}
\nocite{BV12}
\nocite{BV13}
\nocite{DZ95}
\nocite{GPR16}
\nocite{HaGr16}
\nocite{PV07}
\nocite{FIKN06}
\nocite{FA83}

\bibliography{biblio_mkdv}
\bibliographystyle{plain}

\bigskip
\bigskip

\normalsize

\begin{center}
{\scshape Simão Correia}\\
{\footnotesize
Université de Strasbourg\\
CNRS, IRMA UMR 7501\\
F-67000 Strasbourg, France\\
\email{correia@math.unistra.fr}
}

\bigskip

{\scshape Raphaël Côte}\\
{\footnotesize
Université de Strasbourg\\
CNRS, IRMA UMR 7501\\
F-67000 Strasbourg, France\\
\email{cote@math.unistra.fr}
}

\bigskip

{\scshape Luis Vega}\\
{\footnotesize
Departamento de Matemáticas\\
Universidad del País Vasco UPV/EHU\\
Apartado 644, 48080, Bilbao, Spain\\
\ \\
Basque Center for Applied Mathematics BCAM\\
Alameda de Mazarredo 14, 48009 Bilbao, Spain\\
\email{luis.vega@ehu.es}
}
\end{center}

\end{document}